\documentclass[preprint,sort&compress,12pt,draft]{elsarticle}
\usepackage{amsmath}
\usepackage{mathrsfs}
\usepackage{stmaryrd}
\usepackage{bm}
\usepackage{amsfonts}
\usepackage{amssymb}
\usepackage{amsthm}
\usepackage[RGB]{xcolor}
\newcommand{\be}{\begin{equation}}
\newcommand{\bea}{\begin{equation}\begin{aligned}}
\newcommand{\beas}{\begin{equation*}\begin{aligned}}
\newcommand{\eeas}{\end{aligned}\end{equation*}}
\newcommand{\eea}{\end{aligned}\end{equation}}
\newcommand{\ee}{\end{equation}}

\makeatletter
\def\ps@pprintTitle{%
	\let\@evenfoot\@oddfoot
}
\makeatother
\begin{document}
\begin{frontmatter}
%\title{
\title{Global well-posedness for 3D compressible and incompressible micropolar fluids without angular viscosity in strip domains}

%%%%%
\author[sJ]{Youyi Zhao}
\ead{zhaoyouyi957@163.com}
%\author[sJ]{Weiwei Wang\corref{cor1}}
%\cortext[cor1]{Corresponding. }
%\ead{wei.wei.84@163.com}
\address[sJ]{School of Mathematics and Statistics, Fuzhou University, Fuzhou, 350108, China.}
\begin{abstract}
This paper investigates an initial-boundary value problem for three-dimensional (3D) micropolar fluids in a strip domain, including both the compressible and the (homogeneous and inhomogeneous) incompressible cases in the absence of angular viscosity.
The analysis is rendered difficult by two major obstacles: the degeneracy induced by vanishing angular viscosity, and the strong coupling between micro-rotation and velocity fields characterized by a non-dissipative anti-symmetric structure.
Moreover, the presence of physical boundaries in the strip domain further compounds these obstacles.
While the global well-posedness of the 2D incompressible Cauchy problem has been established in the literature,
no results are available for the 3D system and the initial-boundary value problem in both two and three dimensions, particularly in the compressible case.
By exploiting the intrinsic structure of the system and establishing delicate energy estimates, we overcome these difficulties and prove the global well-posedness of strong solutions near equilibrium in a strip domain.
%
%This paper investigates an initial-boundary value problem for three-dimensional (3D) micropolar fluids in a strip domain, including both the compressible and the (homogeneous and inhomogeneous) incompressible cases in the absence of angular viscosity.
%The analysis is rendered difficult by two major obstacles: the degeneracy induced by vanishing angular viscosity, and the strong coupling between micro-rotation and velocity fields characterized by a non-dissipative anti-symmetric structure.
%Moreover, the presence of physical boundaries in the strip domain further compounds these obstacles.
%While the global well-posedness of the 2D incompressible Cauchy problem has been established in the literature,
%no results are available for the initial-boundary value problem in both two and three dimensions, even near equilibrium, particularly in the compressible case.
%By exploiting the intrinsic structure of the system and establishing delicate energy estimates, we overcome these difficulties and prove the global well-posedness of strong solutions near equilibrium in a strip domain.
\end{abstract}

\begin{keyword}
%\MSC[2000] 35Q35\sep  76D03.
%(2000 is the default)
Micropolar fluids; Global well-posedness; Zero angular viscosity; Compressible; Incompressible; Strip domains.
\end{keyword}
\end{frontmatter}

% \linenumbers

\newtheorem{thm}{Theorem}[section]
\newtheorem{lem}{Lemma}[section]
\newtheorem{pro}{Proposition}[section]
\newtheorem{cor}{Corollary}[section]
\newproof{pf}{\emph{Proof}}
\newdefinition{rem}{Remark}[section]
\newtheorem{definition}{Definition}[section]

\section{Introduction}\label{sec:01}
\numberwithin{equation}{section}

%%%%%%%%%%%%%%%%%%%%%%%%%%%%%%%%%%%%%%%%%%%%%%%%%%%%%%%%%%%%%%%%%%%%%%%%%%%%%
\subsection{Background and motivation}\label{0608241825n}

The micropolar fluid model, first proposed by Eringen \cite{Eringen1966} in the 1960s,
describes a class of fluids exhibiting microscopic effects arising from the local structure and micromotions of the fluid particles.
Physically, micropolar fluids may represent media
consisting of the motion of rigid, randomly oriented or spherical particles that have their own spins or microrotations suspended in a viscous medium.
Typical examples include liquid crystals, ferrofluids, colloidal suspensions, and biological fluids such as blood.
Mathematically, the micropolar fluid model introduces a micro-rotation field alongside the standard velocity and pressure, and is accordingly governed by an additional conservation law for angular momentum \cite{Lukaszewicz1999}.
This results in a system with microstructure and non-symmetric stress tensors, fundamentally distinguishing it from standard viscous theories.
More background information on the micropolar fluid model can be found in \cite{Cowin1968,Erdogan1970,Eringen1969} and the references therein.

From a mathematical perspective, the model presents a compelling challenge due to its inherent nonlinearity and strong coupling between the micro-rotation and velocity fields. Over the past decades, significant progress has been made in establishing the global existence and uniqueness of solutions, as well as analyzing long-time behavior, vanishing viscosity limits, and regularity criteria for potential singularities.
Moreover, recent studies have focused on the regularity problem under various settings, including the magneto-micropolar case \cite{lin2020global,FengSCM,zhai2025stability},
the micropolar Rayleigh--B\'enard case \cite{Melkemi2025}, among others.
%the micropolar MHD system with partial dissipation \cite{Zhang2021},
%In terms of numerical schemes, methods such as the two-grid finite element method \cite{Lou2021} and the variational multiscale method \cite{Yuan2022} have been developed to enhance computational efficiency.
%Beyond theoretical and numerical analysis, this model has been applied in engineering problems, notably in modeling unsteady rotating flows to turbomachinery and geophysics \cite{Turkyilmazoglu2020}. Furthermore, approximate analytical solutions derived via differential transforms have offered new insights into 2D flow behaviors %\cite{Dorjsuren2024}.
%%Collectively, these studies reinforce the theoretical foundations and practical utility of micropolar fluid dynamics.

The theory of compressible micropolar viscous fluids provides a framework for describing viscous compressible media containing randomly oriented suspended particles.
The presence of these microparticles induces complex physical phenomena that cannot be adequately characterized by classical compressible viscous barotropic flow models.
Mathematically, the governing system for the compressible (isentropic) micropolar fluids is given as follows:
\begin{equation}\label{010100a}
\begin{cases}
\rho_t +\mathrm{div}(\rho u)=0,\\[1mm]
\rho(u_t + u \cdot \nabla u) + \nabla P(\rho) = (\mu + \chi) \Delta u+(\mu+\mu'- \chi)\nabla\mathrm{div} u + 2\chi \nabla \times w, \\[1mm]
\rho(w_t + u \cdot \nabla w)  + 4\chi w = 2\chi \nabla \times u + \mu_0 \Delta w  + (\mu_0+\mu'_0) \nabla \mathrm{div} w.
\end{cases}
\end{equation}
Here the unknowns $\rho:=\rho(x,t)$, $u := (u_1,u_2,u_3)^\top(x,t)$ and $w := (w_1,w_2, w_3)^\top(x,t)$ represent the density, velocity, and micro-rotation velocity, respectively.
The scalar pressure $P(\rho)$ is a function of the density, which is assumed to be smooth and strictly increasing.
The constants $\mu$, $\mu'$, and $\chi$ denote
the shear viscosity, bulk viscosity, and dynamic micro-rotation viscosity coefficients, respectively, which satisfy the following physical restrictions:
\begin{align*}
    \mu>0,\;\; \chi > 0, \quad 2\mu + 3\mu' - 4\chi \geqslant 0.
\end{align*}
Additionally, the constants $\mu_0$ and $\mu'_0$ represent the angular viscosities.
In the governing system \eqref{010100a},
we refer to \eqref{010100a}$_1$ as the continuity equation,
\eqref{010100a}$_2$ as the momentum equations, and \eqref{010100a}$_3$ as the micro-rotation equations.

The system \eqref{010100a} has been extensively studied regarding the existence, uniqueness, and regularity of solutions.
Mujakovi\'c successively established a series of results concerning the existence and large-time behavior of solutions in the absence of vacuum (i.e., $\rho > 0$), covering both the 1D Cauchy problem \cite{Mujakovic2005} and initial-boundary value problems \cite{Mujakovic1998Local,Mujakovic1998Global,Mujakovic2001}.
For multidimensional analyses, we refer the reader to \cite{Liu2016} and the references therein.
Furthermore, significant progress has been made on the global well-posedness of compressible micropolar fluid system with vacuum (see, e.g., \cite{Chen2015,Hou2021,feng2019global,liuzhong2025global}).
However, the absence of angular viscosity renders the analysis more intricate,
as the lack of the diffusion term $\Delta w$ causes the system to lose its parabolic character, degenerating into a hyperbolic equations.
In particular, these obstacles are further compounded when the physical boundaries are present.
It is worth noting that the stabilizing effect of the magnetic field in magneto-micropolar systems can often compensate for such degeneracy (see, e.g., \cite{FengSCM,zhai2025stability}).
Despite these related studies, to our best knowledge, no results are available for the global well-posedness of the pure compressible micropolar fluid system in the absence of angular viscosity, even for small initial data in the initial-boundary value setting.
Motivated by this, in this paper, we first consider such a system in a strip domain $\Omega:=\mathbb{R}^2\times(0,1)$, governed by:
\begin{equation}\label{0101}
\begin{cases}
\rho_t +\mathrm{div}(\rho u)=0\quad &\mbox{in } \Omega,\\[0.5mm]
\rho(u_t + u \cdot \nabla u) + \nabla P(\rho) = (\mu + \chi) \Delta u+(\mu+\mu'- \chi)\nabla\mathrm{div} u + 2\chi \nabla \times w  &\mbox{in } \Omega, \\[0.5mm]
\rho(w_t + u \cdot \nabla w)  + 4\chi w = 2\chi \nabla \times u \quad &\mbox{in } \Omega.
\end{cases}
\end{equation}
System \eqref{0101} admits an equilibrium state given by $r_c:=(\bar{\rho},0,0)$, where $\bar{\rho}$ is a positive constant.

Our analysis naturally extends to the incompressible case. If the fluid is incompressible, then the velocity is divergence-free and the pressure becomes an unknown variable, denoted by ${p}:={p}(x,t)$.
Consequently, system \eqref{0101} reduces to the following 3D inhomogeneous incompressible micropolar fluid system:
\begin{equation}\label{0102}
\begin{cases}
\rho_t +u\cdot \nabla\rho=0\quad &\mbox{in } \Omega,\\
\rho(u_t + u \cdot \nabla u) + \nabla {p} = (\mu + \chi) \Delta u + 2\chi \nabla \times w \quad &\mbox{in } \Omega, \\
\rho(w_t + u \cdot \nabla w)  + 4\chi w = 2\chi \nabla \times u \quad &\mbox{in } \Omega, \\
\mathrm{div}u=0 \quad &\mbox{in } \Omega.
\end{cases}
\end{equation}
Similarly, system \eqref{0102} admits an equilibrium state $r_{iic}:=(\bar{\rho},0,0)$ with the same positive constant density $\bar{\rho}$ for notational convenience.
%Note that the system \eqref{0102} admits an equilibrium state $r_{iic}:=(\bar{\rho},0,0,\bar{p})$, where %$\bar{\rho}$ be a positive constant, and
%$$\bar{\rho}\;\mbox{ is a positive constant},\quad\mbox{and}\quad\nabla\bar{p}=0.$$

When the density is constant and the angular viscosity is included in \eqref{0102} (i.e., the homogeneous incompressible micropolar fluid system with full viscosities), Galdi and Rionero \cite{12} established the global existence and uniqueness of weak and strong solutions to the Cauchy problem.
See also \cite{ Lukaszewicz1999} for the existence of weak solutions.
Chen and Miao \cite{chen2012global} further proved global well-posedness in critical Besov spaces in $\mathbb{R}^3$. Subsequent investigations have explored various aspects, including large-time behavior \cite{6, 28}, vanishing viscosity limits \cite{3, 5, 43}, and fractional dissipation \cite{7, 38, 41}.
However, the analysis becomes more delicate when the angular viscosity is absent.
For the 2D homogeneous incompressible system, Dong and Zhang \cite{8} proved the global existence and uniqueness of smooth solutions in $\mathbb{R}^2$ by introducing a novel auxiliary quantity.
In the context of homogeneous incompressible magneto-micropolar fluids, Lin and Xiang \cite{lin2020global} established the global well-posedness of small solutions under a uniform horizontal magnetic field in a 2D strip domain.
Despite these advances, the global well-posedness of the 3D incompressible micropolar system without angular viscosity remains an open problem.
Indeed, to our best knowledge, even global results for such initial-boundary value problems has not been established.

In this paper, we establish the global well-posedness of micropolar fluid systems without angular viscosity in the strip domain $\mathbb{R}^2\times(0,1)$, covering both compressible flows and (homogeneous and inhomogeneous) incompressible flows.
By exploiting the intrinsic structure of the system and establishing delicate energy estimates combined with refined trace estimates, we overcome the analytical difficulties induced by the absence of angular viscosity, the strong coupling between the micro-rotation and velocity fields, as well as the presence of boundaries,
thereby proving the global well-posedness of strong solutions near equilibrium.
Our results may shed new light on the global dynamics of micropolar fluids in the context of initial-boundary value problems.

%%%%%%%%%%%%%%%%%%%%%%%%%%%%%%%%%%%%%%%%%%%%%%%%%%%%%%%%%%%%%%%%%%%%%%%%%%%%%
\subsection{Perturbation}\label{0608241825nmn}
This paper aims to establish the global well-posedness of system~\eqref{0101} near the equilibrium state $r_c := (\bar{\rho}, 0, 0)$, as well as that of system~\eqref{0102} near $r_{iic} := (\bar{\rho}, 0, 0)$.
To achieve this,
we first reformulate both systems in terms of perturbations from equilibrium so as to facilitate the subsequent analysis.

For the compressible case, let the perturbation around $r_c$  be denoted by
$$q=\rho-\bar{\rho},\quad u=u-0,\quad\mbox{and}\quad w=w-0.$$
Then, system \eqref{0101} can be rewritten as
\begin{equation}\label{0101n}
\begin{cases}
q_t +u \cdot \nabla q+(\bar{\rho}+q)\mathrm{div}u=0\quad &\mbox{in } \Omega,\\[0.5mm]
(\bar{\rho}+q)(u_t + u \cdot \nabla u) +  P'(\bar{\rho})\nabla q - (\mu + \chi) \Delta u-(\mu+\mu'- \chi)\nabla\mathrm{div} u \quad\\
=2\chi \nabla \times w-\nabla\mathcal{N}_{P}
 \quad &\mbox{in } \Omega, \\[0.5mm]
(\bar{\rho}+q)(w_t + u \cdot \nabla w)  + 4\chi w = 2\chi \nabla \times u \quad &\mbox{in } \Omega,
\end{cases}
\end{equation}
where the remainder term $\mathcal{N}_{P}$ is defined via the Taylor expansion:
\begin{align*}
\mathcal{N}_{P}:=P(\bar{\rho}+q)-P(\bar{\rho})- P'(\bar{\rho}) q=\int_{\bar{\rho}}^{\bar{\rho}+q}(\bar{\rho}+q-z)P''(z)\mathrm{d}z.
\end{align*}
To complete the statement of the problem,
the perturbed system \eqref{0101n} is supplemented with the following initial-boundary value conditions:
\begin{align}
&\label{0101b}
(q,u,w)\big|_{t=0}=(q^0,u^0,w^0)\quad\mbox{in } \Omega,\\[0.5mm]
&\label{0101c}
u\big|_{\partial\Omega}=0\quad\mbox{for any}\;t>0,
\end{align}
where $\partial\Omega$ is regarded as the boundary of $\Omega$, namely $\partial\Omega:=\mathbb{R}^2\times\{0,1\}$.
Therefore,  proving the global well-posedness of system~\eqref{0101} near the equilibrium state $r_c$
reduces to establishing the global well-posedness of the problem \eqref{0101n}--\eqref{0101c}.

Turning to the incompressible case, we similarly define
the perturbation around $r_{iic}$ by
$$q=\rho-\bar{\rho},\quad u=u-0,\quad w=w-0,$$
%\quad\mbox{and}\quad p=\tilde{p}-\bar{p},$$
the system \eqref{0102} then becomes
\begin{equation}\label{0102n}
\begin{cases}
q_t +u \cdot \nabla q=0\quad &\mbox{in } \Omega,\\[0.5mm]
(\bar{\rho}+q)(u_t + u \cdot \nabla u) +  \nabla p - (\mu + \chi) \Delta u =2\chi \nabla \times w
 \quad &\mbox{in } \Omega, \\[0.5mm]
(\bar{\rho}+q)(w_t + u \cdot \nabla w)  + 4\chi w = 2\chi \nabla \times u \quad &\mbox{in } \Omega,\\[0.5mm]
\mathrm{div}u=0 \quad &\mbox{in } \Omega.
\end{cases}
\end{equation}
We impose the following initial-boundary value conditions
\begin{align}
&\label{0102b}
(q,u,w)\big|_{t=0}=(q^0,u^0,w^0)\quad\mbox{in }\Omega,\\[0.5mm]
&\label{0102c}
u\big|_{\partial\Omega}=0\quad\mbox{for any}\;t>0.
\end{align}
In particular, assuming the density is constant (denoted by $\rho=1$ for notational convenience),
the problem \eqref{0102n}--\eqref{0102c} reduces to the following 3D homogeneous incompressible micropolar system:
\begin{equation}\label{0103}
\begin{cases}
u_t + u \cdot \nabla u + \nabla p -(\mu + \chi) \Delta u = 2\chi \nabla \times w \qquad &\mbox{in } \Omega, \\[0.5mm]
w_t + u \cdot \nabla w  + 4\chi w = 2\chi \nabla \times u  \quad &\mbox{in } \Omega,\\[0.5mm]
\mathrm{div} u=0\quad &\mbox{in } \Omega,\\[0.5mm]
(u,w)\big|_{t=0}=(u^0,w^0)\quad &\mbox{in } \Omega, \\[0.5mm]
u=0\quad &\mbox{on } \partial\Omega.
\end{cases}
\end{equation}
Consequently, proving the global well-posedness of system~\eqref{0102} near the equilibrium state $r_{iic}$
is equivalent to establishing the global well-posedness of problem \eqref{0102n}--\eqref{0102c}, which in turn implies the global well-posedness of the homogeneous incompressible problem \eqref{0103}.

\vspace{3mm}
The rest of this paper is organized as follows.
In Section \ref{main results}, we present our main results, including the global well-posedness of solutions for the compressible problem \eqref{0101n}--\eqref{0101c} and the incompressible problem \eqref{0102n}--\eqref{0102c}, as stated in Theorems \ref{thm1} and \ref{thm2}, respectively.
Sections \ref{2025thm1} and \ref{2025thm2} are devoted to the detailed proofs of Theorems \ref{thm1} and \ref{thm2}.

%%%%%%%%%%%%%%%%%%%%%%%%%%%%%%%%%%%%%%%%%%%%%%%%%%%%%%%%%%%%%%%%%%%%%%%%%%%%%%%%%%%%%%%%%%%%%%%%%%%%%%%%%%%%%%%%%%%%%
\section{Main results}\label{main results}

%%%%%%%%%%%%%%%%%%%%%%%%%%%%%%%%%%%%%%%%%%%%%%%%%%%%%%%%%%%%%%%%%%%%%%%%%%%%%%%%%%%%%%%%%%%%%%%%%%%%%%%%%%%%%%%%%%%%%%
%\subsection{Notations}
Before stating our result, we list some conventions for notation which will be frequently used in this paper.

(1) Basic notations: $\mathbb{R}^+_0:=[0,\infty)$;
$\bar{\Omega}:=\mathbb{R}^2\times[0,1]$;
$\int\cdot\mathrm{d}x:=\int_{\Omega}\cdot\mathrm{d}x$ denotes the integral over domain $\Omega$;
$a\lesssim b$ means that $a\leqslant cb$ for some ``universal'' constant $c>0$, where the constant $c$ may depend on
the physical parameters such as $\mu$, $\mu'$ and $\bar{\rho}$, and may be different from line to line.
Moreover,
$f_{\mathrm{h}}:=(f_1, f_2)$
for $f=(f_{\mathrm{h}}, f_3)^\top$;
$\mathrm{div}_{\mathrm{h}}f_{\mathrm{h}}:=\partial_1f_1+\partial_2f_2$, and
$\mathrm{d}x_{\mathrm{h}}:=\mathrm{d}x_{1}\mathrm{d}x_{2}$.
%The notation $\partial_{i}$ ($1\leqslant i\leqslant3$) denotes the partial derivative with respect to the $i$-th coordinate of variable $x$, i.e., $\partial_{i}:=\partial_{x_i}$.
Letters $\lambda_{i}$ ($1\leqslant i\leqslant14$), $c_{j}$ ($1\leqslant j\leqslant8$) and $\delta_{k}$ ($1\leqslant k\leqslant3$) are fixed constants which may depend on the parameters.
Let $\mathbb{N} = \{0, 1, 2, \dots\}$ be the collection of non-negative integers. When using space-time differential multi-indices, we write $\mathbb{N}^{1+d} = \{\alpha = (\alpha_0, \alpha_1, \dots, \alpha_d)\}$ to emphasize that the $0$--index term is related to temporal derivatives. For just spatial derivatives, we write $\mathbb{N}^d$. For $\alpha \in \mathbb{N}^{1+d}$, we write $\partial^\alpha = \partial_t^{\alpha_0} \partial_1^{\alpha_1} \cdots \partial_d^{\alpha_d}$, and define the parabolic counting by $|\alpha| = 2\alpha_0 + \alpha_1 + \cdots + \alpha_d$.
Particularly, if $\alpha_0=0$ and $d=2$, we denote
$\partial_{\mathrm{h}}^{\alpha}:=\partial_1^{\alpha_1}\partial_2^{\alpha_2}$ for some multiindex of order $\alpha:=(\alpha_1, \alpha_2)$ with $|\alpha|=\alpha_1+\alpha_2$;
and $\partial_{\mathrm{h}}^{i}$ denotes $\partial_{\mathrm{h}}^{\alpha}$ for any $\alpha$ satisfying $|\alpha|=i$.
For any differential operator $\partial^{\alpha}$, we use the commutator:
$$[\partial^{\alpha},f]g = \partial^{\alpha}(fg) - f\partial^{\alpha}g.$$

(2) Simplified notations of function spaces:
\begin{equation*}
\begin{aligned}
&L^p:=L^p (\Omega)=W^{0,p}(\Omega),
%\;\;W^{i,2}:=W^{i,2}(\Omega),
\;\;\; H^i:=W^{i,2}(\Omega),\\[1mm]
&H_0^1:=\{f\in H^1~|~f|_{\partial\Omega}=0\;\mbox{in the sense of trace}\},\;\;
H_0^j:=H_0^1\cap H^j,
\end{aligned}
\end{equation*}
where $1<p\leqslant\infty$, and $i\geqslant0$, $j\geqslant1$ are integers.

(3) Simplified Sobolev norms and semi-norms:
\begin{equation*}
\begin{aligned}
&\|\cdot\|_{i}:=\|\cdot\|_{W^{i,2}(\Omega)},\;\;
\|\cdot\|_{i,k}:=\sum_{|\alpha|=i} \|\partial_{\mathrm{h}}^{\alpha}\cdot\|_{k},\;\;
\|\cdot\|_{\underline{i},k}:=\sum_{j=0}^i\|\cdot\|_{j,k},
\end{aligned}
\end{equation*}
where $i$, $j$ and $k$ are non-negative integers.
For convenience, we sometimes use the simplified norm
$\|(\aleph_{1}f_1,\ldots,\aleph_{n}f_n)\|_{X}:=
\sqrt{\sum_{1\leqslant i\leqslant n}\|\aleph_{i} f_i\|_{X}^2}$, where $f_1$, $\ldots$, $f_n$ are vector or scalar functions that may have different dimensions.
Each of $\aleph_{1},\ldots,\aleph_{n}$ is a differential operator or the identity.

%%%%%%%%%%%%%%%%%%%%%%%%%%%%%%%%%%%%%%%%%%%%%%%%%%%%%%%%%%%%%%%%%%%
%\subsection{Main results}\label{subsec:03}
\vspace{5mm}
We first state our global well-posedness result for the compressible problem \eqref{0101n}--\eqref{0101c}.
To this end, we define the energy functional $\mathcal{E}(t)$ and dissipation functional $\mathcal{D}(t)$:
\begin{align}
\mathcal{E}(t):=&
\|(q,u,w)\|_{2}^2+\|(q_{t},w_{t})\|_{1}^2+\|u_t\|_{0}^2 \nonumber,\\[1.5mm]
\mathcal{D}(t):=&
\|u\|_{3}^2+\|w\|_{2}^2+\|\nabla q\|_{1}^2+\|(q_{t},u_{t},w_{t})\|_{1}^2.\nonumber
%+{\color{red}\|(q_{t}+u\cdot \nabla q,w_{t}+u\cdot \nabla w)\|_{2}^2}.\nonumber
\end{align}
Our global well-posedness result for \eqref{0101n}--\eqref{0101c} is given as follows.
\begin{thm}\label{thm1}
There exists a sufficiently small constant $\delta_1>0$ such that, if the initial data $(q^0, u^0, w^0)\in H^{2}\times H^{2}_{0}\times H^{2}$ satisfies
$${\|( q^0 ,u^0,w^0)\|_{2}^2}\leqslant\delta_1. $$
Then the initial-boundary value  problem \eqref{0101n}--\eqref{0101c} admits a unique strong solution
$(q, u,w)\in C^{0}(\mathbb{R}^+_0, H^{2} \times H^{2}_{0} \times H^{2})$;
moreover, the solution enjoys the estimate
\begin{align}\label{202109090810}
&\mathcal{E} (t) + \int_{0}^{t} {\mathcal{D}}(\tau)\mathrm{d}\tau\lesssim \|( q^0, u^0, w^0)\|_{2}^2.
\end{align}
\end{thm}

\begin{rem}\label{rem:1}
%The global well-posedness result in Theorem \ref{thm1} also holds for the case in a horizontally periodic domain $\tilde{\Omega}$, in which we can show that the solution decays exponentially in time, provided that $\int_{\tilde{\Omega}}q^0\mathrm{d}x=0$.
The global well-posedness result in Theorem \ref{thm1} also holds for the case on the half-space $\mathbb{R}_{+}^3:=\{x=(x_1,x_2,x_3)^{\top}\in\mathbb{R}^3: x_3>0\}$ (with slight modifications).
Similarly, the result for the homogeneous incompressible case (see Corollary \ref{corthm2} below) remains valid in this setting.
\end{rem}
\begin{rem}\label{rem:12}   %\cite{JJTIWYHCMP}
Our approach in the proof of Theorem \ref{thm1} draws inspiration from \cite{FengSCM} and \cite{MANTIC481} to establish the energy-dissipation estimates for $w$ and $q$, respectively. It is worth noting that the lack of angular viscosity and the presence of microstructure coupling (via $\nabla\times w$ and $\nabla\times u$) lead to a loss of derivatives, posing substantial obstacles to the analysis.
Moreover, compared to the work in \cite{FengSCM} formulated in Lagrangian coordinates, the pressure term here introduces additional difficulties in establishing estimates for the normal derivatives of $w$. This difficulty stems from the absence of the magnetic tension, which would otherwise support the dissipation mechanism in conjunction with the pressure in Lagrangian coordinates.
\end{rem}

We now present our global well-posedness result for the incompressible problem \eqref{0102n}--\eqref{0102c}.
For this purpose, we define the energy and dissipation functionals:
\begin{align}
\mathscr{E}(t):=&\|q\|_{2}^2+\mathfrak{E}(t) \nonumber,\\[1.5mm]
\mathfrak{E}(t):=&
\|(u,w)\|_{2}^2+\|\nabla p\|_{0}^2+\|u_t\|_{0}^2+\|w_{t}\|_{1}^2 \nonumber,\\[1.5mm]
\mathfrak{D}(t):=&
\|u\|_{3}^2+\|w\|_{2}^2+\|\nabla p\|_{1}^2+\|(u_{t},w_t)\|_{1}^2.\nonumber
\end{align}
Our global well-posedness result for the problem \eqref{0102n}--\eqref{0102c} is given as follows.
\begin{thm}\label{thm2}
There exists a sufficiently small constant $\delta_2>0$ such that, if the initial data $(q^0, u^0, w^0)\in H^{2}\times H^{2}_{0}\times H^{2}$ satisfies
$\mathrm{div}u^0=0$ in $\Omega$ and
$${\|( q^0, u^0, w^0)\|_{2}^2}\leqslant\delta_2. $$
Then the initial-boundary value  problem \eqref{0102n}--\eqref{0102c} admits a unique strong solution
$(q, u,w)\in C^{0}(\mathbb{R}^+_0,$ $H^{2} \times H^{2}_{0} \times H^{2}$) with a unique associated pressure $p$ (up to a constant) satisfying $\nabla p\in C^{0}(\mathbb{R}^+_0, L^{2})$.
Moreover, the solution enjoys
\begin{enumerate}[\quad (1)]
\item[(1)] the exponential decay-in-time estimate:
\begin{align}
&\label{202609090810}
e^{ct}\mathfrak{E}(t)+\int_{0}^{t}e^{c\tau/2}\mathfrak{D}(\tau)\mathrm{d}\tau
 \lesssim \|(u^0, w^0)\|_{2}^2;
\end{align}
\item[(2)] the global estimate:
\begin{align}
&\label{202609090810nm}
 \mathscr{E}(t)\lesssim \|( q^0, u^0, w^0)\|_{2}^2.
\end{align}
\end{enumerate}
\end{thm}

\begin{rem}\label{rem:21}
It is important to note that in the incompressible system, the pressure is not determined by the density. Since the pressure is now a new unknown without an ODE structure as in the compressible case, we employ the regularity theory of the Stokes problem to derive estimates for the pressure.
\end{rem}

%\begin{rem}\label{rem:13}
%{\color{blue}
%We remark that the key strategy for the proof of Theorems \ref{thm1} and \ref{thm2} lies in converting vertical estimates into horizontal ones; namely, closing the normal estimates by using tangential estimates. Indeed, such a strategy has proven effective in studying the viscous and non-resistive MHD system as well as the viscous surface wave problem. We refer to the works by \cite{FengSCM,JJTIWYHCMP, WYJAIM2020} for examples. However, the vorticity terms in the present case introduce additional difficulties.
%}
%\end{rem}

When the density $\rho\equiv 1$, the counterpart of Theorem \ref{thm2} for the homogeneous incompressible case reads as follows.
\begin{cor}\label{corthm2}
There exists a sufficiently small constant $\delta_3>0$ such that, if the initial data $(u^0, w^0)\in H^{2}_{0}\times H^{2}$ satisfies
$\mathrm{div}u^0=0$ in $\Omega$ and
$${\|(u^0, w^0)\|_{2}^2}\leqslant\delta_3. $$
Then the initial-boundary value  problem \eqref{0103} admits a unique strong solution
$(u,w)\in C^{0}(\mathbb{R}^+_0,$ $H^{2}_{0} \times H^{2}$) with a unique associated pressure $p$ (up to a constant) satisfying $\nabla p\in C^{0}(\mathbb{R}^+_0, L^{2})$;
moreover, the solution enjoys the estimate
\begin{align}\label{202604120810nm}
&e^{ct}\mathfrak{E} (t) + \int_{0}^{t}e^{ct/2} {\mathfrak{D}}(\tau)\mathrm{d}\tau\lesssim \|(u^0, w^0)\|_{2}^2.
\end{align}
\end{cor}

\vspace{3mm}
We now briefly sketch the proof strategy for Theorems \ref{thm1} and \ref{thm2}.
The proof is based on the local well-posedness of solutions, the \emph{a priori} estimates, and the standard continuity argument.
Since the local well-posedness of problems \eqref{0101n}--\eqref{0101c} and \eqref{0102n}--\eqref{0102c} can be established via  iteration schemes (see e.g., \cite{MAJBAL,kawashima1984systems,GYTILW1,JJTIIAWangC}), the key step is to derive the \emph{a priori} estimates.
We begin with the proof strategy for Theorem \ref{thm1}.

First, we consider the basic energy identities for the problem \eqref{0101n}--\eqref{0101c}, obtained by testing the continuity, momentum, and micro-rotation equations with  $q$, $u$ and $w$, respectively:
\begin{align}
&\label{2026nm12271537}
\frac{1}{2}\frac{\mathrm{d}}{\mathrm{d}t}\left(\int P'(\bar{\rho})\bar{\rho}^{-1}|q|^2\mathrm{d}x+\int\bar{\rho}|u|^2\mathrm{d}x\right)\nonumber\\[1mm]
&+(\mu+\chi)\int|\nabla u|^2\mathrm{d}x+(\mu+\mu'-\chi)\int|\mathrm{div}u|^2 \mathrm{d}x=2\chi\int\nabla \times w\cdot u\mathrm{d}x+N^1,\\[1mm]
&\label{2026nghb1540}
\frac{1}{2}\frac{\mathrm{d}}{\mathrm{d}t}\int\bar{\rho}|w|^2\mathrm{d}x
+4\chi\int|w|^2 \mathrm{d}x=2\chi\int\nabla \times u\cdot w\mathrm{d}x+N^2,\qquad\qquad\qquad\qquad
\end{align}
where $N^1$ and $N^2$ represent the integrals involving the nonlinear terms of $(q,u,w)$.
Compared with the classical compressible Navier-Stokes equations, the energy identities for the compressible micropolar fluid system introduce two additional terms involving vorticities. These terms are linear and anti-symmetric, posing significant challenges due to the lack of angular viscosity. However, we observe that the shear viscosity term generates additional velocity dissipation, which helps control these obstacle terms.
Indeed, it is easy to verify that
\begin{align*}%\label{2026asdg}
\int\nabla \times w\cdot u\mathrm{d}x=\int\nabla \times u\cdot w\mathrm{d}x,
\end{align*}
which allows us to deduce the following dissipation positivity:
\begin{align*}
\|(\nabla u,w)\|_{0}^2\lesssim&(\mu+\chi)\int|\nabla u|^2\mathrm{d}x+(\mu+\mu'-\chi)\int|\mathrm{div}u|^2 \mathrm{d}x+4\chi\int|w|^2 \mathrm{d}x\nonumber\\[1mm]
&-2\chi\int\nabla \times w\cdot u\mathrm{d}x
-2\chi\int\nabla \times u\cdot w\mathrm{d}x.
\end{align*}
Consequently, we can reduce \eqref{2026nm12271537}--\eqref{2026nghb1540} to the following $L^2$-inequality:
\begin{align}
&\label{2026nm1227nm1537}
\frac{1}{2}\frac{\mathrm{d}}{\mathrm{d}t}\left(\int P'(\bar{\rho})\bar{\rho}^{-1}|q|^2\mathrm{d}x+\int\bar{\rho}|u|^2\mathrm{d}x
+\int\bar{\rho}|w|^2\mathrm{d}x\right)
+c\|(\nabla u,w)\|_{0}^2\lesssim \big|N^1+N^2\big|.
\end{align}

To control the nonlinear terms and establish the existence of strong solutions, we must work with higher-order energy functionals; specifically, normal-derivative estimates for $(q,u,w)$ are required.
However, since \eqref{2026nm1227nm1537} does not apply to the higher-order normal spatial derivatives of $(q,u,w)$
in Sobolev norms due to the presence of physical boundaries, our proof requires several steps. Roughly speaking, the argument is divided into the tangential estimates and normal estimates for $(q, u, w)$, all of which are established under the \emph{a priori} assumption \eqref{aprpioses}.

Observe that tangential derivatives (including temporal and horizontal ones) naturally preserve the boundary condition,
which allows for integration by parts and the application of the Friedrichs inequality \eqref{friedrich} in the energy evolution.
Hence, the strategy employed to establish \eqref{2026nm1227nm1537} can also be applied to the tangential energy evolution (see Proposition \ref{pro12281500}).
Having established the tangential estimates for $(q,u,w)$,
we turn to the estimates of the normal derivatives. This constitutes a subtle and significant part of the overall proof.
Motivated by \cite{MANTIC481}, we first derive the estimates for $\partial_3q$  by exploiting its energy-dissipation structure.
Specifically, by combining the continuity equation with the third component of the momentum equations, we eliminate $\partial_{3}^2u_{3}$ to obtain:
\begin{align*}
\partial_{3}(q_t+u\cdot\nabla q)+c\partial_{3}q=g,
\end{align*}
where $g$ is mainly composed of $\partial_tu_3+\partial_j\partial_{\mathrm{h}}u_i+(\nabla\times w)_3+\emph{h.o.t}$, and $\emph{h.o.t}$ denotes the nonlinear terms involving $(q,u,w)$ (see \eqref{12281700}).
This equation can be viewed as an evolution equation for $\partial_3q$, which exhibits a structure analogous to the dissipative ODE $\partial_tf+cf=g$, modulo error terms, and thus displays a favorable energy-dissipation structure.
Since $(\nabla\times w)_3=\partial_1w_2-\partial_2w_1$, the linear terms in $g$ involve only tangential derivatives of $(\nabla u,u,w)$.
This fundamental relation consequently suggests that the higher-order estimates of $\partial_3q$ can be reduced to the tangential estimates; see Lemma \ref{lem:dfifessimM} for details.

Following the strategy in \cite{FengSCM} regarding the $\mathrm{div}$-$\mathrm{curl}$ decomposition for $w$,
we now establish the normal estimates for $(\nabla\times w,\mathrm{div}w)$, which simultaneously yields the dissipation estimates for $\partial_3^2u$.
We first derive the evolution equation for $\nabla\times w$ by applying the curl operator to the  micro-rotation equations.
By combining this with the momentum equations, we then obtain two ODE systems that decouple the third component from the horizontal components of the equations for $\partial_3^2u$ and $\nabla \times w$ (see \eqref{2026195637}--\eqref{2026195638}).
Let $W=(W_{\mathrm{h}},W_3)^\top:=\nabla\times w$. To illustrate the argument, we focus on the horizontal components for the pair $(\partial_{3}^2u,W)$:
\begin{equation*}%\label{2026195637n260421m}
\begin{cases}
(\mu + \chi)\partial_{3}^2u_{\mathrm{h}}+2\chi W_{\mathrm{h}} = \tilde{\mathcal{H}}_{\mathrm{h}} \quad (\text{where } \tilde{\mathcal{H}}_{\mathrm{h}} \sim \partial_tu_{\mathrm{h}} + \nabla_{\mathrm{h}} q + \nabla\partial_{\mathrm{h}} u + \emph{h.o.t}) & \mbox{in } \Omega ,\\[1.5mm]
\bar{\rho}\big(\partial_{t}W_{\mathrm{h}}+u\cdot\nabla W_{\mathrm{h}} \big)+4\chi W_{\mathrm{h}}+2\chi\partial_{3}^2u_{\mathrm{h}} = \tilde{\mathcal{G}}_{\mathrm{h}} \quad (\text{where } \tilde{\mathcal{G}}_{\mathrm{h}} \sim \nabla\partial_{\mathrm{h}} u + \emph{h.o.t}) & \mbox{in } \Omega.
\end{cases}
\end{equation*}
The presence of the term  $\mu \partial_{3}^2u_{\mathrm{h}}$ allows us to establish positive dissipation for $(\partial_{3}^2u_{\mathrm{h}},W_{\mathrm{h}})$ from the coupling between $\partial_{3}^2u_{\mathrm{h}}$ and $W_{\mathrm{h}}$; see the derivation of \eqref{20926504113339nm} for details.
It is worth noting that aside from the pressure term $\nabla_{\mathrm{h}} q$, the linear terms in $\tilde{\mathcal{H}}_{\mathrm{h}}$ and $\tilde{\mathcal{G}}_{\mathrm{h}}$
involve only tangential or first-order normal derivative of $u$.
This leads us to anticipate that the higher-order estimates for $(\partial_{3}^2u_{\mathrm{h}},W_{\mathrm{h}})$ can also be reduced to their tangential estimates.
However, since we possess dissipation only for $\partial_3q$ (but not $\nabla_{\mathrm{h}}q$) and cannot apply the Stokes estimate to derive bounds for $\nabla_{\mathrm{h}}q$ due to the lack of control over $W$ established previously, the pressure term $\nabla_{\mathrm{h}} q$ presents a significant challenge:
$$\int\partial_{\mathrm{h}}^{j}\partial_{3}^{1-i}\nabla_{\mathrm{h}}q\cdot\partial_{\mathrm{h}}^{j}\partial_{3}^{1-i}\partial_3^2u_{\mathrm{h}}\mathrm{d}x.$$
This differs markedly from the setting in \cite{FengSCM}, where magnetic tension supports the dissipation mechanism alongside the pressure in Lagrangian coordinates.
To circumvent this difficulty, we analyze three cases based on $0\leqslant j\leqslant i\leqslant1$. The most delicate case occurs when $i=j=1$:
$$\int\partial_{\mathrm{h}}\nabla_{\mathrm{h}}q\cdot\partial_{\mathrm{h}}\partial_3^2u_{\mathrm{h}}\mathrm{d}x.$$
Given the presence of sufficient horizontal derivatives,
we address this integral by integrating by parts and applying the refined trace estimate \eqref{37190928}, which yields the bound $\|\nabla q\|_{1}\|u\|_{3}^{1/2}\|u\|_{\underline{2},1}^{1/2}$ (see \eqref{20260411na}).
This integral is finally closed via Young's inequality combined with the tangential energy estimate and the bound for $\|u\|_{3}+\|\nabla q\|_{1}$.
Similarly, in deriving the normal estimates for $(\partial_3^2u_3, W_3)$, we encounter a comparable difficulty:
$$\int\partial_{\mathrm{h}}^{j}\partial_{3}^{1-i}\partial_3q\partial_{\mathrm{h}}^{j}\partial_{3}^{1-i}\partial_3^2u_3\mathrm{d}x.$$
This integral can be controlled by invoking Young's inequality, combined with the estimate for $\partial_3q$ established in Lemma \ref{lem:dfifessimM} and the bound for $\|\partial_3^{3-i}u\|_{\underline{i},0}$.
For the detailed derivation of the estimates for $(\partial_3^2u,\nabla\times w)$,  we refer to Lemma \ref{uwnormal}.

Additionally, by applying the divergence operator to the micro-rotation equations, we derive the evolution equation for $\mathrm{div} w$,
This equation also exhibits a structure analogous to the dissipative ODE $\partial_tf+cf=g$, modulo error terms.
Hence, the energy-dissipation estimate for $\mathrm{div} w$ follows readily; see Lemma \ref{divwnormal} for further details.
Therefore, collecting the tangential and normal estimates of $(q,u,w)$ recursively, we obtain
\begin{align*}
&\frac{\mathrm{d}}{\mathrm{d}t}\tilde{\mathcal{E}} + {\mathcal{D}}\leqslant0,
\end{align*}
where the energy functional $\tilde{\mathcal{E}}$ is equivalent to ${\mathcal{E}}$.
Consequently, the desired \emph{a priori} estimate can be proven.
Based on the \emph{a priori} estimate, together with the local well-posedness result and a standard continuity argument, we can promptly establish Theorem \ref{thm1}.
The detail proof will be presented in Section \ref{2025thm1}.

We now outline the proof strategy for Theorem \ref{thm2}. Analogous to the proof of Theorem \ref{thm1}, the analysis is divided into tangential and normal estimates for $(u, w)$, all established under the \emph{a priori} assumption \eqref{aprpioseses}.
Furthermore, we adopt the same strategy used in Theorem \ref{thm1} to derive the normal estimates for $w$.
A key difference lies in the treatment of the pressure $p$.
Since $p$ is now a new unknown lacking an ODE structure as in the compressible case,
the only viable approach to derive estimates for $p$ is to employ the elliptic regularity theory of the Stokes problem, which simultaneously yields estimates for $u$.
In this framework, the divergence-free condition $\mathrm{div}u=0$ is essential.
This condition also allows us to express $\partial_3u_3$ as $-\mathrm{div}_{\mathrm{h}}u_{\mathrm{h}}$.
For instance, by invoking this condition, we transform the estimate for $\|\partial_3^{3}u_3\|_{0}$ into $\|\partial_3^2\mathrm{div}_{\mathrm{h}}u_{\mathrm{h}}\|_{0}$, thereby converting normal estimates into horizontal ones, and simplifying the estimates for $W_3$.
Conversely, when deriving normal estimates for $W_{\mathrm{h}}$, we face an integral difficulty similar to that of the compressible case:
$$\int\partial_{\mathrm{h}}^{j}\partial_{3}^{1-i}\nabla_{\mathrm{h}}p\cdot\partial_{\mathrm{h}}^{j}\partial_{3}^{1-i}\partial_3^2u_{\mathrm{h}}\mathrm{d}x.$$
The most delicate cases arise when $i=j=1$ and $i=j=0$.
For the case $i=j=1$, we follow the same argument as in the compressible case to close the estimates using the bound $\|\nabla p\|_{1}\|u\|_{3}^{1/2}\|u\|_{\underline{2},1}^{1/2}$.
For the case $i=j=0$, we utilize the Stokes estimate to establish regularity for $\|\nabla\partial_{\mathrm{h}}p\|_{0}$
(see \eqref{11292030es}).
This allows us to circumvent the difficulty by transforming the estimate for
$\|(\partial_3^3u,\partial_3W)\|_{0}$ into the estimate for $\|\partial_{\mathrm{h}}(\partial_3^2u,W)\|_{0}$, utilizing the bound on $\|\nabla\partial_{\mathrm{h}}p\|_{0}$ and Young's inequality.
Moreover, we have the same energy-dissipation estimate for $\mathrm{div}w$ as in the compressible case.
Consequently, we obtain
\begin{align*}
&\frac{\mathrm{d}}{\mathrm{d}t}\tilde{\mathfrak{E}} + {\mathfrak{D}}\leqslant0,
\end{align*}
where the energy functional $\tilde{\mathfrak{E}}$ is equivalent to $\mathfrak{E}$.
In particular,
\begin{equation*}
\tilde{\mathfrak{E}}\lesssim \mathfrak{D}.
\end{equation*}
Accordingly, the exponential time-decay of $\mathfrak{E}$ and the integral time-decay of $\mathfrak{D}$ are established.
The estimate for $\|q\|_{2}^2$ in the energy then follows via Gronwall's inequality and the integral time-decay of $\mathfrak{D}$.
Finally, equipped with these \emph{a priori} estimates, the local well-posedness result, and a standard continuity argument, we complete the proof of Theorem \ref{thm2}. The detailed proof is presented in Section \ref{2025thm2}.

%%%%%%%%%%%%%%%%%%%%%%%%%%%%%%%%%%%%%%%%%%%%%%%%%%%%%%%%%%%%%%%%%%%%%%%%%%%%%%%%%%%%%%%%%%%%%%%%%%%%%%%%%%%%%%%%55
\section{Proof of Theorem \ref{thm1}} \label{2025thm1}
This section is devoted to the proof of Theorem \ref{thm1}, the key step in the proof of Theorem \ref{thm1} is to derive the \emph{a priori} energy estimates.
To this end, let $(q,u,w)$ be a solution to the problem \eqref{0101n}--\eqref{0101c} defined on $\Omega\times [0,T]$ with $T>0$.
We make the \emph{a priori} assumption that
\begin{align}
&\label{aprpioses}
\sup_{t\in[0,T]}\left\|(q, u, w)(t)\right\|_{2}< \delta\in(0,1),
\end{align}
where  $\delta$ is a sufficiently small constant.
Furthermore, we assume that the solution possesses proper regularity so that the procedure of formal calculations makes sense.
In view of \eqref{aprpioses} and the Sobolev embedding inequality \eqref{embed2}, it is straightforward to deduce that
\begin{align}\label{202604121856nh}
\bar{\rho}/2\leqslant \inf_{x\in\Omega}\{\bar{\rho}+q\}\leqslant\sup_{x\in\Omega}\{\bar{\rho}+q\}\leqslant 3\bar{\rho}/2.
\end{align}
Next, we shall first establish some preliminary estimates that will be instrumental in deriving the energy evolution.
Subsequently, we derive the \emph{a priori}  energy inequalities for the solutions.

%%%%%%%%%%%%%%%%%%%%%%%%%%%%%%%%%%%%%%%%%%%%%%%%%%%%%%%%%
\subsection{Preliminary estimates}
First,
in order to use the nonlinear structure of \eqref{0101n} to derive the tangential energy estimates
for the temporal and horizontal spatial derivatives of the solution, we apply $\partial^{\alpha}$ with $\alpha\in\mathbb{N}^{1+2}$ to the problem \eqref{0101n}--\eqref{0101c} to obtain:
\begin{equation}\label{20260412ne}
\begin{cases}
\partial^{\alpha}q_t +u \cdot \nabla \partial^{\alpha}q+\bar{\rho}\mathrm{div}\partial^{\alpha}u=\mathcal{N}^{1,\alpha}\quad &\mbox{in } \Omega,\\[1mm]
(\bar{\rho}+q)(\partial^{\alpha}u_t + u \cdot \nabla \partial^{\alpha}u)
 - (\mu + \chi) \Delta \partial^{\alpha}u-(\mu+\mu'- \chi)\nabla\mathrm{div} \partial^{\alpha}u\quad \\[1mm]
=- P'(\bar{\rho})\nabla \partial^{\alpha}q+2\chi \nabla \times \partial^{\alpha}w+\mathcal{N}^{2,\alpha}+\mathcal{N}^{3,\alpha}
 \quad &\mbox{in } \Omega, \\[1mm]
(\bar{\rho}+q)\big(\partial^{\alpha}w_t + u \cdot \nabla \partial^{\alpha}w\big)  + 4\chi \partial^{\alpha} w = 2\chi \nabla \times \partial^{\alpha}u
+\mathcal{N}^{4,\alpha} \quad &\mbox{in } \Omega,\\[1mm]
\partial^{\alpha}u=0 &\mbox{on } \partial\Omega,
\end{cases}
\end{equation}
where we have defined that
\begin{align*}
&\mathcal{N}^{1,\alpha}:=-[\partial^{\alpha},u ]\cdot \nabla q -\partial^{\alpha}(q\mathrm{div}u),\\[1mm]
&\mathcal{N}^{2,\alpha}:=-(\bar{\rho}+q)[\partial^{\alpha},u ]\cdot \nabla u ,\\[1mm]
&\mathcal{N}^{3,\alpha}:=-[\partial^{\alpha},q](u_t + u \cdot \nabla u)-\partial^{\alpha}\nabla\mathcal{N}_{P},\\[1mm]
&\mathcal{N}^{4,\alpha}:=-(\bar{\rho}+q)[\partial^{\alpha},u ]\cdot \nabla w-[\partial^{\alpha},q](w_t+u\cdot \nabla w).
\end{align*}

When using the linear structure of \eqref{0101n}, it is more convenient to formulate it as a perturbation of the
linearized system:
\begin{equation}\label{0101nb}
\begin{cases}
q_t +\bar{\rho}\mathrm{div}u=\mathcal{N}^5\quad &\mbox{in } \Omega,\\[0.5mm]
\bar{\rho}u_t +  P'(\bar{\rho})\nabla q - (\mu + \chi) \Delta u-(\mu+\mu'- \chi)\nabla\mathrm{div} u-2\chi \nabla \times w\quad\\
=\mathcal{N}^6+\mathcal{N}^7 \quad &\mbox{in } \Omega, \\[0.5mm]
\bar{\rho}w_t  + 4\chi w- 2\chi \nabla \times u=\mathcal{N}^8
 \quad &\mbox{in } \Omega,\\[0.5mm]
u=0 \quad &\mbox{on } \partial\Omega,\quad
\end{cases}
\end{equation}
where we have defined that
\begin{align*}
&\mathcal{N}^5:=-u \cdot \nabla q-q\mathrm{div}u,\\[1mm]
&\mathcal{N}^6:=
-\bar{\rho} u \cdot \nabla u-q(u_t + u \cdot \nabla u),\quad\mathcal{N}^7:=-\nabla\mathcal{N}_{P},\\[1mm]
&\mathcal{N}^8:=-\bar{\rho}u \cdot \nabla w-q(w_t + u \cdot \nabla w).
\end{align*}
We now establish estimates for these nonlinear terms, which will be used in the energy evolution.
\begin{lem}
\label{201806291049}
Under assumption \eqref{aprpioses} with sufficiently small $\delta$, the following estimates hold:
\begin{enumerate}
\item[(1)] Estimates for $\mathcal{N}^{k,\alpha}$ with $1\leqslant k\leqslant 4$:
\begin{align}
\label{12272120}
&\|(\mathcal{N}^{1,\alpha},\mathcal{N}^{2,\alpha},\mathcal{N}^{4,\alpha}) \|_0\lesssim\sqrt{\mathcal{E}\mathcal{D}}
\quad\;\mathrm{for}\;\;|\alpha|\leqslant2,\\[1mm]
\label{12272120n}
&\|\mathcal{N}^{3,\alpha} \|_0\lesssim\sqrt{\mathcal{E}\mathcal{D}}\quad\;\;\mathrm{for}\;\;|\alpha|\leqslant1\;\mathrm{or}\;\alpha_0=1.
\end{align}
\item[(2)] Estimates for $\mathcal{N}^{k}$ with $5\leqslant k\leqslant 8$:
\begin{align}
& \label{1227213738}
%\sum_{k=5}^{8}\|\mathcal{N}^{k}\|_{i}
\|(\mathcal{N}^{6},\mathcal{N}^{7})\|_{i}\lesssim\|(q,u)\|_{2}\big(\|u\|_{2+i}+\|(\nabla w,\nabla q)\|_{i}\big)
\quad\mathrm{for}\;\;0\leqslant i\leqslant1,\\[1mm]
& \label{1227213738n}
\|(\mathcal{N}^{5},\mathcal{N}^{8})\|_{i}\lesssim\|(q,u)\|_{2}\big(\|(u,w)\|_{1+i}+\|\nabla q\|_{i}\big)
\quad\mathrm{for}\;\;0\leqslant i\leqslant1.
\end{align}
\item[(3)]
Let $\mathcal{N}^{9}$ and $\mathcal{N}^{10}$ be defined by
\begin{align}
&\label{20260412a}
\mathcal{N}^{9}:=-\nabla\times\big(q( w_{t}+u\cdot\nabla w)\big)-\bar{\rho}[\nabla\times, u]\cdot\nabla w,\\[1mm]
&\label{20260412b}
\mathcal{N}^{10}:=-\mathrm{div}\big(q( w_{t}+u\cdot\nabla w)\big)-\bar{\rho}[\mathrm{div}, u]\cdot\nabla w.
\end{align}
Then we have
\begin{align}
& \|\mathcal{N}^{9}\|_1+\|\mathcal{N}^{10}\|_1\lesssim \sqrt{\mathcal{E}\mathcal{D}}. \label{12281530}
\end{align}
\end{enumerate}
\end{lem}
\begin{pf}
By virtue of \eqref{aprpioses}--\eqref{202604121856nh} and the product estimate \eqref{product}, we readily obtain \eqref{12272120n} as well as the following estimate:
\begin{align}
&\label{202604122002}
\|\mathcal{N}^7\|_{i}\lesssim\|q\|_{2}\|\nabla q\|_{i}\quad\mathrm{for}\;\;0\leqslant i\leqslant1.
\end{align}

Moreover, in view of  \eqref{aprpioses}--\eqref{202604121856nh} and \eqref{202604122002},
it follows directly from  \eqref{0101n}$_2$--\eqref{0101n}$_3$ that
\begin{align}
&\label{202604122032}
\|(u_{t}+u\cdot\nabla u)\|_{i}\lesssim\|u\|_{2+i}+\|(\nabla q,\nabla w)\|_{i}\quad\mathrm{for}\;\;0\leqslant i\leqslant1,\\[1mm]
&\label{202604122036}
\|(w_{t}+u\cdot\nabla w)\|_{j}\lesssim\|(\nabla u,w)\|_{j}\quad\mathrm{for}\;\;0\leqslant j\leqslant2.
\end{align}
Finally, exploiting \eqref{aprpioses}--\eqref{202604121856nh}, \eqref{202604122002}--\eqref{202604122036} and \eqref{product}, we arrive at \eqref{12272120}, \eqref{1227213738}--\eqref{1227213738n} and \eqref{12281530}.
This completes the proof.
\hfill $\Box$
\end{pf}

%%%%%%%%%%%%%%%%%%%%%%%%%%%%%%%%%%%%%%%%%%%%%%%%%%%%%%%%%%%%%%%%%%%%%%%%%
\subsection{Tangential estimates of $(q, u, w)$}
In this subsection, we are devoted to deriving the tangential estimates of $(q,u,w)$.
We define the tangential energy by
\begin{align} \label{20260412nh}
\bar{\mathcal{E}}:=\sum_{j=0}^{1}\|\partial_t^j(q,u,w)\|_{\underline{2-j},0}^2
\end{align}
and the corresponding tangential dissipation by
\begin{align} \label{20260412nhn}
\bar{\mathcal{D}}:=\sum_{j=0}^{1}\left(\|\partial_t^ju\|_{\underline{2-j},1}^2+\|\partial_t^jw\|_{\underline{2-j},0}^2\right).
\end{align}
The tangential energy evolution for $(q,u,w)$ is established as follows.
\begin{pro}\label{pro12281500}
Under assumption \eqref{aprpioses} with sufficiently small $\delta$,
there exists a function $\tilde{\bar{\mathcal{E}}}$, which is equivalent to $\bar{\mathcal{E}}$, such that
\begin{align} \label{12281503nmnm}
&\frac{\mathrm{d}}{\mathrm{d}t}\tilde{\bar{\mathcal{E}}}+c\bar{\mathcal{D}}\lesssim \sqrt{\mathcal{E}}\mathcal{D}.
\end{align}
\end{pro}
\begin{pf}
Let $\alpha\in\mathbb{N}^{1+2}$ satisfy $|\alpha|\leqslant2$.
Taking the inner product of \eqref{20260412ne}$_2$ with $\partial^{\alpha}u$  in $L^2$,
integrating by parts over $\Omega$, and then using the boundary condition \eqref{20260412ne}$_4$, we have
\begin{align}
&\label{202312271537}
\frac{1}{2}\frac{\mathrm{d}}{\mathrm{d}t}\int(\bar{\rho}+q)|\partial^{\alpha}u|^2\mathrm{d}x
+(\mu+\chi)\int|\nabla \partial^{\alpha}u|^2\mathrm{d}x+(\mu+\mu'-\chi)\int|\mathrm{div}\partial^{\alpha}u|^2 \mathrm{d}x\nonumber\\[1.5mm]
&=\int P'(\bar{\rho})\partial^{\alpha}q\mathrm{div} \partial^{\alpha}u\mathrm{d}x
+2\chi\int\nabla \times \partial^{\alpha}w\cdot \partial^{\alpha}u\mathrm{d}x
+\int(\mathcal{N}^{2,\alpha}+\mathcal{N}^{3,\alpha})\cdot \partial^{\alpha}u \mathrm{d}x\nonumber\\[1mm]
&\quad+\frac{1}{2}\int \big((\bar{\rho}+q)_{t}+\mathrm{div}((\bar{\rho}+q) u)\big)|\partial^{\alpha}u|^2\mathrm{d}x.
\end{align}
By virtue of \eqref{0101n}$_1$, we obtain
\begin{align*}%\label{202604091147nm}
\frac{1}{2}\int \big((\bar{\rho}+q)_{t}+\mathrm{div}((\bar{\rho}+q) u)\big)|\partial^{\alpha}u|^2\mathrm{d}x=0.
\end{align*}
Thanks to \eqref{20260412ne}$_1$ and  \eqref{20260412ne}$_4$, we deduce that
\begin{align*}
%&\label{202312271124}
&\int P'(\bar{\rho})\partial^{\alpha}q\mathrm{div}\partial^{\alpha}u\mathrm{d}x\\[1mm]
&
=-\frac{1}{2}\frac{\mathrm{d}}{\mathrm{d}t}\int P'(\bar{\rho})\bar{\rho}^{-1}|\partial^{\alpha}q|^2\mathrm{d}x
+\frac{1}{2}\int P'(\bar{\rho})\bar{\rho}^{-1}\mathrm{div}u|\partial^{\alpha}q|^2\mathrm{d}x
+\int P'(\bar{\rho})\bar{\rho}^{-1}\mathcal{N}^{1,\alpha}\partial^{\alpha}q\mathrm{d}x.
\end{align*}
Substituting the above two identities into \eqref{202312271537} yields
\begin{align}\label{2026a1537}
\frac{1}{2}&\frac{\mathrm{d}}{\mathrm{d}t}\left(\int(\bar{\rho}+q)|\partial^{\alpha}u|^2\mathrm{d}x
+\int P'(\bar{\rho})\bar{\rho}^{-1}|\partial^{\alpha}q|^2\mathrm{d}x\right)\nonumber\\[1.5mm]
&+(\mu+\chi)\int|\nabla \partial^{\alpha}u|^2\mathrm{d}x+(\mu+\mu'-\chi)\int|\mathrm{div}\partial^{\alpha}u|^2 \mathrm{d}x
-2\chi\int\nabla \times \partial^{\alpha}w\cdot \partial^{\alpha}u\mathrm{d}x\nonumber\\[1.5mm]
=&\int P'(\bar{\rho})\bar{\rho}^{-1}\mathcal{N}^{1,\alpha}\partial^{\alpha}q\mathrm{d}x
+\int(\mathcal{N}^{2,\alpha}+\mathcal{N}^{3,\alpha})\cdot \partial^{\alpha}u \mathrm{d}x
+\frac{1}{2}\int P'(\bar{\rho})\bar{\rho}^{-1}\mathrm{div}u|\partial^{\alpha}q|^2\mathrm{d}x.
\end{align}

In analogy with the derivation of \eqref{2026a1537}, taking the inner product of \eqref{20260412ne}$_3$ with $\partial^{\alpha}w$  in $L^2$,
integrating by parts over $\Omega$, and invoking \eqref{0101n}$_1$ and \eqref{20260412ne}$_4$ again, it follows that
\begin{align}\label{2026b1140}
&\frac{1}{2}\frac{\mathrm{d}}{\mathrm{d}t}\int(\bar{\rho}+q) |\partial^{\alpha}w|^2\mathrm{d}x
+4\chi\int|\partial^{\alpha}w|^2 \mathrm{d}x
-2\chi\int\nabla \times \partial^{\alpha}u\cdot \partial^{\alpha}w\mathrm{d}x
%\nonumber\\[1mm]&
=\int\mathcal{N}^{4,\alpha}\cdot  \partial^{\alpha}w \mathrm{d}x.
\end{align}
Adding \eqref{2026b1140} to \eqref{2026a1537} then gives rise to
\begin{align}\label{2026c1537}
\frac{1}{2}&\frac{\mathrm{d}}{\mathrm{d}t}\left(\int(\bar{\rho}+q)(|\partial^{\alpha}u|^2+|\partial^{\alpha}w|^2)\mathrm{d}x
+\int P'(\bar{\rho})\bar{\rho}^{-1}|\partial^{\alpha}q|^2\mathrm{d}x\right)\nonumber\\[1.5mm]
&+(\mu+\chi)\int|\nabla \partial^{\alpha}u|^2\mathrm{d}x+(\mu+\mu'-\chi)\int|\mathrm{div}\partial^{\alpha}u|^2 \mathrm{d}x
+4\chi\int|\partial^{\alpha}w|^2 \mathrm{d}x\nonumber\\[1.5mm]
&-2\chi\int\nabla \times \partial^{\alpha}w\cdot \partial^{\alpha}u\mathrm{d}x
-2\chi\int\nabla \times \partial^{\alpha}u\cdot \partial^{\alpha}w\mathrm{d}x\nonumber\\[1.5mm]
=&\int P'(\bar{\rho})\bar{\rho}^{-1}\mathcal{N}^{1,\alpha}\partial^{\alpha}q\mathrm{d}x
+\int\mathcal{N}^{2,\alpha}\cdot \partial^{\alpha}u \mathrm{d}x
+\int\mathcal{N}^{3,\alpha}\cdot \partial^{\alpha}u \mathrm{d}x
+\int\mathcal{N}^{4,\alpha}\cdot  \partial^{\alpha}w \mathrm{d}x
\nonumber\\[1.5mm]
&+\frac{1}{2}\int P'(\bar{\rho})\bar{\rho}^{-1}\mathrm{div}u|\partial^{\alpha}q|^2\mathrm{d}x:=\sum_{k=1}^{5}I_k.
\end{align}

By virtue of the boundary condition $\partial^{\alpha}u\big|_{\partial\Omega}=0$ and the following two vector identities
$$\nabla \times w\cdot u=\nabla \times u\cdot w+\mathrm{div}(w\times u),\quad\Delta u=\nabla\mathrm{div}u-\nabla\times\nabla\times u,$$
one can readily verify that
\begin{align}\label{2026g}
&\int\nabla \times \partial^{\alpha}w\cdot \partial^{\alpha}u\mathrm{d}x\nonumber\\[1mm]
&=\int\nabla \times \partial^{\alpha}u\cdot \partial^{\alpha}w\mathrm{d}x
+\int\mathrm{div}( \partial^{\alpha}w\times \partial^{\alpha}u)\mathrm{d}x
=\int\nabla \times \partial^{\alpha}u\cdot \partial^{\alpha}w\mathrm{d}x\quad
\end{align}
and
\begin{align}\label{2026h}
&\int|\nabla \times \partial^{\alpha}u-2 \partial^{\alpha}w|^2\mathrm{d}x\nonumber\\[1mm]
&=\int(|\nabla \partial^{\alpha}u|^2-|\mathrm{div}\partial^{\alpha}u|^2 +4|\partial^{\alpha}w|^2 )\mathrm{d}x
-4\int\nabla \times \partial^{\alpha}u\cdot \partial^{\alpha}w\mathrm{d}x.\qquad\quad
\end{align}
Invoking these two significant identities, we can then update \eqref{2026c1537} as
\begin{align}\label{2026k1537}
&\frac{1}{2}\frac{\mathrm{d}}{\mathrm{d}t}\left(\int(\bar{\rho}+q)(|\partial^{\alpha}u|^2+|\partial^{\alpha}w|^2)\mathrm{d}x
+\int P'(\bar{\rho})\bar{\rho}^{-1}|\partial^{\alpha}q|^2\mathrm{d}x\right)\nonumber\\[1.5mm]
&+\mu\int|\nabla \partial^{\alpha}u|^2\mathrm{d}x+(\mu+\mu')\int|\mathrm{div}\partial^{\alpha}u|^2 \mathrm{d}x
+\chi\int|\nabla \times \partial^{\alpha}u-2 \partial^{\alpha}w|^2\mathrm{d}x
\nonumber\\[1mm]&
=\sum_{k=1}^{5}I_k.
\end{align}

Furthermore, we note that $\|\nabla\times \partial^{\alpha}u\|_{0}\leqslant\|\nabla \partial^{\alpha}u\|_{0}$ due to $\partial^{\alpha}u\big|_{\partial\Omega}=0$, and that
\begin{align*}
\|\partial^{\alpha}w\|_{0}^2\lesssim\|\nabla \times \partial^{\alpha}u-2 \partial^{\alpha}w\|_{0}^2+\epsilon \|\nabla\times \partial^{\alpha}u\|_{0}\quad\mbox{holds for any}\;\epsilon>0.
\end{align*}
Consequently, there exists a suitably small constant $\lambda_1> 0$ such that
\begin{align}\label{2026l1537}
&\frac{1}{2}\frac{\mathrm{d}}{\mathrm{d}t}\left(\int(\bar{\rho}+q)(|\partial^{\alpha}u|^2+|\partial^{\alpha}w|^2)\mathrm{d}x
+\int P'(\bar{\rho})\bar{\rho}^{-1}|\partial^{\alpha}q|^2\mathrm{d}x\right)\nonumber\\[1mm]
&+\lambda_1\left(\int|\nabla \partial^{\alpha}u|^2\mathrm{d}x+\int|\partial^{\alpha}w|^2\mathrm{d}x\right)
\lesssim\sum_{k=1}^{5}I_k.
\end{align}

We now proceed to estimate $I_k$ in sequence.
First, observe that for any multi-index $\alpha\in\mathbb{N}^{1+2}$ with $|\alpha|\leqslant2$, the operator
$\partial^{\alpha}$ corresponds to either a time derivative $\partial_{t}$ or a horizontal spatial derivative
$\partial_{\mathrm{h}}^{i}\;(0\leqslant i\leqslant 2)$. By applying \eqref{12272120}, we easily obtain the bound:
\begin{align}\label{2026d1538}
I_1+I_2+I_4\lesssim \sqrt{\mathcal{E}}\mathcal{D}.
\end{align}

Next, we turn our attention to $I_3$. When $|\alpha|\leqslant1$ or $\alpha_0=1$, it follows from \eqref{12272120n} that
\begin{align}\label{2026f1538}
I_3\lesssim \sqrt{\mathcal{E}}\mathcal{D}.
\end{align}
For the case where $|\alpha|=2$ and $\alpha_0=0$, we first integrate by parts in $x_{\mathrm{h}}$ and then use \eqref{202604122002}--\eqref{202604122032} with $i=1$ to deduce that
\begin{align}\label{2026g1538}
I_3\lesssim\|\nabla q\|_{1}\|(u_{t}+u\cdot\nabla u)\|_{1}\|\partial_{\mathrm{h}}^2u\|_{1}
+\|\mathcal{N}^7\|_{1}\|\partial_{\mathrm{h}}^3u\|_{0}\lesssim \sqrt{\mathcal{E}}\mathcal{D}.
\end{align}

Finally, regarding $I_5$, if $|\alpha|=0$, we employ integration by parts to have
\begin{align}\label{2026e1538nmnm}
I_5=-\int P'(\bar{\rho})\bar{\rho}^{-1}qu\cdot\nabla q\mathrm{d}x
\lesssim \|q\|_{2}\|(u,\nabla q)\|_{0}^2\lesssim \sqrt{\mathcal{E}}\mathcal{D}.
\end{align}
In the remaining case where $|\alpha|\neq0$, an application of \eqref{embed2} yields
\begin{align}\label{2026e1538}
I_5\lesssim \|u\|_{3}\|\partial^{\alpha}q\|_{0}^2\lesssim \sqrt{\mathcal{E}}\mathcal{D}.
\end{align}

Defining
\begin{align}\label{2026gh1538}
\tilde{\bar{\mathcal{E}}}:=\sum_{|\alpha|\leqslant2}\left(\int(\bar{\rho}+q)(|\partial^{\alpha}u|^2+|\partial^{\alpha}w|^2)\mathrm{d}x
+\int P'(\bar{\rho})\bar{\rho}^{-1}|\partial^{\alpha}q|^2\mathrm{d}x\right).
\end{align}
In view of \eqref{202604121856nh}, we easily verify the equivalence between $\tilde{\bar{\mathcal{E}}}$ and $\bar{\mathcal{E}}$.
Thus, inserting \eqref{2026d1538}--\eqref{2026e1538} into \eqref{2026l1537}, summing over $\alpha$, and employing \eqref{friedrich} for $\partial^{\alpha}u$ together with \eqref{2026gh1538}, we arrive at the desired estimate \eqref{12281503nmnm}.
This completes the proof.
\hfill$\Box$
\end{pf}

%%%%%%%%%%%%%%%%%%%%%%%%%%%%%%%%%%%%%%%%%%%%%%%%%%%%%%%%%%%%%%%%%%%%%%%%%
\subsection{Normal estimates of $(q, u, w)$}

In this subsection, we aim to establish the estimates for the normal derivatives of $(q, u, w)$.
More precisely, we exploit the ODE structure of the linear perturbed model and the elliptic structure of the momentum equation to derive bounds on the normal derivatives of the solution. Subsequently, we apply the regularity theory of elliptic problem to obtain further estimates for $u$ in the energy.

{\bf{ Estimates for $q$.}}
In the spirit of \cite{MANTIC481}, we first define the quantity $Q$ by
$Q:=q_{t}+u\cdot\nabla q$.
From \eqref{0101nb}$_1$ we deduce that $Q$ satisfies
\begin{align}\label{202604101017}
\partial_{3}Q+\bar{\rho}\partial_{3}^2u_3=-\bar{\rho}\partial_{3}\mathrm{div}_{\mathrm{h}}u_{\mathrm{h}}-q\mathrm{div}u.
\end{align}
Combining \eqref{202604101017} with the third component of  \eqref{0101nb}$_2$ to eliminate $\partial_{3}^2u_{3}$ yields
\begin{align}\label{12281700}
\bar{\rho}^{-1}(2\mu+\mu')\partial_{3}Q+{P}'(\bar{\rho})\partial_{3}q=\mathcal{L}+\mathfrak{N},
\end{align}
where
\begin{align*}
&\mathcal{L}:=-\bar{\rho}\partial_{t}u_{3}-(\mu+\chi)\partial_{3}\mathrm{div}_{\mathrm{h}}u_{\mathrm{h}}+(\mu+\chi)\Delta_{\mathrm{h}}u_3+2\chi (\nabla\times w)_{3},\\[1.5mm]
&\mathfrak{N}:=\big(\mathcal{N}^{6}+\mathcal{N}^7\big)_{3}-\bar{\rho}^{-1}(2\mu+\mu')q\mathrm{div}u.
\end{align*}
We are now in a position to derive the following estimates for $\partial_3q$.
\begin{lem}\label{lem:dfifessimM}
Under assumption \eqref{aprpioses} with sufficiently small $\delta$, it holds that
\begin{align}\label{202112152140}
&\frac{\mathrm{d}}{\mathrm{d}t}\|\partial_{3}^{2-i}q\|_{\underline{i},0}^2
+c\|\partial_{3}^{2-i}q\|_{\underline{i},0}^2\nonumber\\[1mm]
&\lesssim\bar{\mathcal{D}}
+\left\|(\partial_3^{2-i}\partial_{\mathrm{h}}u,\partial_3^{1-i}(\nabla\times w)_3)\right\|_{\underline{i},0}^2+\sqrt{\mathcal{E}}\mathcal{D}\quad\;\;\mathrm{for}\;\;0\leqslant i\leqslant 1.
\end{align}
\end{lem}
\begin{pf}
Taking the norm $\|\partial_3^{1-i}\cdot\|_{\underline{i},0}$ of both sides of \eqref{12281700}, we obtain
\begin{align}\label{202112152142}
\|\partial_3^{1-i}\big(\bar{\rho}^{-1}(2\mu+\mu')\partial_{3}Q+{P}'(\bar{\rho})\partial_{3}q\big)\|_{\underline{i},0}^2
=\|\partial_3^{1-i}\big(\mathcal{L}+\mathfrak{N}\big)\|_{\underline{i},0}^2.
\end{align}
Observing that the left-hand side can be expanded as
\begin{align}\label{202604101050}
&\|\partial_3^{1-i}\big(\bar{\rho}^{-1}(2\mu+\mu')\partial_{3}Q+{P}'(\bar{\rho})\partial_{3}q\big)\|_{\underline{i},0}^2\nonumber\\[2mm]
&=\|\bar{\rho}^{-1}(2\mu+\mu')\partial_{3}^{2-i}Q\|_{\underline{i},0}^2
+\|{P}'(\bar{\rho})\partial_{3}^{2-i}q\|_{\underline{i},0}^2
\nonumber\\[1.5mm]&\quad\;
+2{P}'(\bar{\rho})\bar{\rho}^{-1}(2\mu+\mu')\sum_{0\leqslant j\leqslant i}\int\partial_{\mathrm{h}}^{j}\partial_{3}^{2-j}Q\partial_{\mathrm{h}}^{j}\partial_{3}^{2-j}q\mathrm{d}x.\quad\quad\;\;
\end{align}
Recalling the quantity $Q:=q_{t}+u\cdot\nabla q$, we deduce that the cross term satisfies
\begin{align}\label{202604101052}
&2{P}'(\bar{\rho})\bar{\rho}^{-1}(2\mu+\mu')\sum_{0\leqslant j\leqslant i}
\int\partial_{\mathrm{h}}^{j}\partial_{3}^{2-j}Q\partial_{\mathrm{h}}^{j}\partial_{3}^{2-j}q\mathrm{d}x\nonumber\\[1mm]
&=2{P}'(\bar{\rho})\bar{\rho}^{-1}(2\mu+\mu')\sum_{0\leqslant j\leqslant  i}\int\partial_{\mathrm{h}}^{j}\partial_{3}^{2-j}(q_t+u\cdot\nabla q)
\partial_{\mathrm{h}}^{j}\partial_{3}^{2-j}q\mathrm{d}x\nonumber\\[1mm]
&={P}'(\bar{\rho})\bar{\rho}^{-1}(2\mu+\mu')\frac{\mathrm{d}}{\mathrm{d}t}\|\partial_{3}^{2-i}q\|_{\underline{i},0}^2
-{P}'(\bar{\rho})\bar{\rho}^{-1}(2\mu+\mu')\sum_{0\leqslant j\leqslant i}\int\mathrm{div}u|\partial_{\mathrm{h}}^{j}\partial_{3}^{2-j}q|^2\mathrm{d}x\nonumber\\[1mm]
&\quad
+ 2{P}'(\bar{\rho})\bar{\rho}^{-1}(2\mu+\mu')\sum_{0\leqslant j\leqslant i}\int[\partial_{\mathrm{h}}^{j}\partial_{3}^{2-j},u]\cdot\nabla q
\partial_{\mathrm{h}}^{j}\partial_{3}^{2-j}q\mathrm{d}x,
\end{align}
where we have used the fact that
$$\int u\cdot\nabla|\partial_{\mathrm{h}}^{j}\partial_{3}^{2-j}q|^2\mathrm{d}x
=-\int\mathrm{div}u|\partial_{\mathrm{h}}^{j}\partial_{3}^{2-j}q|^2\mathrm{d}x.$$
Combining \eqref{202112152142}--\eqref{202604101052} then gives rise to
\begin{align}\label{202112152250}
&2{P}'(\bar{\rho})\bar{\rho}^{-1}(2\mu+\mu')\frac{\mathrm{d}}{\mathrm{d}t}\|\partial_{3}^{2-i}q\|_{\underline{i},0}^2
+\|{P}'(\bar{\rho})\partial_{3}^{2-i}q\|_{\underline{i},0}^2
+\|\bar{\rho}^{-1}(2\mu+\mu')\partial_{3}^{2-i}Q\|_{\underline{i},0}^2\nonumber\\[1mm]
&\lesssim\|\partial_3^{1-i}\big(\mathcal{L}+\mathfrak{N}\big)\|_{\underline{i},0}^2+\|u\|_{3}\|\partial_{3}q\|_{1}^2.
\end{align}

Furthermore, invoking \eqref{aprpioses} and \eqref{1227213738}, we estimate the right-hand side as
\begin{align*}
\|\partial_3^{1-i}\big(\mathcal{L}+\mathfrak{N}\big)\|_{\underline{i},0}^2
&\lesssim\|\partial_3^{1-i}\mathcal{L}\|_{\underline{i},0}^2+\|(\mathcal{N}^6,\mathcal{N}^7)\|_{1}^2+\|q\mathrm{div}u\|_{1}^2\nonumber\\[1mm]
&\lesssim\|\nabla\partial_3^{1-i}\partial_{\mathrm{h}}u\|_{\underline{i},0}^2+\|\partial_3^{1-i}(\nabla\times w)_3\|_{\underline{i},0}^2
+\|u_{t}\|_{1}^2+\sqrt{\mathcal{E}}\mathcal{D}.
\end{align*}
Substituting this bound into \eqref{202112152250} yields \eqref{202112152140}. This completes the proof.
\hfill$\Box$
\end{pf}

{\bf{ Estimates for $w$.}}
Following the approach of \cite{FengSCM}, we first denote
$$W=(W_1,W_2,W_3)^{\top}:=\nabla\times w.$$
Applying the curl operator to \eqref{0101n}$_3$ and using the identity $\Delta u= \nabla\mathrm{div}u-\nabla\times\nabla\times u$, we obtain
\begin{align}\label{202604101338}
\bar{\rho}\big(W_{t}+u\cdot\nabla W \big)+4\chi W=2\chi(\nabla\mathrm{div}u-\Delta u)+\mathcal{N}^9,
\end{align}
where $\mathcal{N}^9$ is defined in \eqref{20260412a}.

In view of \eqref{0101nb}$_2$ and \eqref{202604101338}, we observe that $(\partial_3^2 u_{\mathrm{h}},W_{\mathrm{h}})$ satisfies the system
\begin{equation}\label{2026195637}
\begin{cases}
(\mu + \chi)\partial_{3}^2u_{\mathrm{h}}+2\chi W_{\mathrm{h}}\\
\;=\bar{\rho}\partial_tu_{\mathrm{h}} + P'(\bar{\rho})\nabla_{\mathrm{h}} q
- (\mu + \chi) \Delta_{\mathrm{h}} u_{\mathrm{h}}-(\mu+\mu'- \chi)\nabla_{\mathrm{h}}\mathrm{div} u
-(\mathcal{N}^6+\mathcal{N}^7)_{\mathrm{h}}:=\tilde{\mathcal{H}}_{\mathrm{h}}
 &\mbox{in } \Omega ,\\[1mm]
\bar{\rho}\big(\partial_{t}W_{\mathrm{h}}+u\cdot\nabla W_{\mathrm{h}} \big)+4\chi W_{\mathrm{h}}+2\chi\partial_{3}^2u_{\mathrm{h}}
=2\chi(\nabla_{\mathrm{h}}\mathrm{div}u-\Delta_{\mathrm{h}} u_{\mathrm{h}})+\mathcal{N}^9_{\mathrm{h}}:=\tilde{\mathcal{G}}_{\mathrm{h}} &\mbox{in } \Omega,
\end{cases}
\end{equation}
while $(\partial_3^2 u_{3},W_{3})$ satisfies
\begin{equation}\label{2026195638}
\begin{cases}
(2\mu + \mu')\partial_{3}^2u_{3}+2\chi W_{3}\\
\;=\bar{\rho}\partial_tu_{3}  +  P'(\bar{\rho})\partial_{3} q
- (\mu + \chi) \Delta_{\mathrm{h}} u_{3}-(\mu+\mu'- \chi)\partial_3\mathrm{div}_{\mathrm{h}}u_{\mathrm{h}} - (\mathcal{N}^6+\mathcal{N}^7)_{3}:=\tilde{\mathcal{H}}_{3}
 &\mbox{in } \Omega ,\\[1mm]
\bar{\rho}\big(\partial_{t}W_{3}+u\cdot\nabla W_{3} \big)+4\chi W_{3}
=2\chi(\partial_3\mathrm{div}_{\mathrm{h}}u_{\mathrm{h}}-\Delta_{\mathrm{h}} u_{3})+\mathcal{N}^9_{3}:=\tilde{\mathcal{G}}_{3} &\mbox{in } \Omega.
\end{cases}
\end{equation}
Based on these systems, we are now ready to derive the estimates for $W$ and $\partial_3^2u$.
\begin{lem}\label{uwnormal}
Under assumption \eqref{aprpioses} with sufficiently small $\delta$, it holds that, %for $0\leqslant i\leqslant 1$,
\begin{align}\label{202604102028}
&\frac{\mathrm{d}}{\mathrm{d}t}\overline{\|\partial_{3}^{1-i}W\|}_{\underline{i},0}^2
+c\big(\|\partial_{3}^{1-i}W\|_{\underline{i},0}^2+\|\partial_{3}^{3-i}u\|_{\underline{i},0}^2\big)\nonumber\\[1mm]
&\lesssim \bar{\mathcal{D}}+  \|(\partial_{3}^{2-i}q,\partial_{3}^{2-i}\partial_{\mathrm{h}} u)\|_{\underline{i},0}^2
+(1-i)\|W\|_{1,0}^2
+\|\nabla q\|_{1}\|u\|_{3}^{1/2}\|u\|_{\underline{2},1}^{1/2}+\sqrt{\mathcal{E}}\mathcal{D},
\end{align}
where $\overline{\|\partial_{3}^{1-i}W\|}_{\underline{i},0}^2$ is equivalent to $\|\partial_{3}^{1-i}W\|_{\underline{i},0}^2$ and $0\leqslant i\leqslant 1$.
\end{lem}
\begin{pf}
Let $0\leqslant j\leqslant i\leqslant1$.
Applying $\partial_{\mathrm{h}}^{j}\partial_{3}^{1-i}$ to \eqref{2026195638}$_1$ and \eqref{2026195638}$_2$,
and taking the inner product of the resulting identities with $\partial_{\mathrm{h}}^{j}\partial_{3}^{3-i}u_3$ and $\partial_{\mathrm{h}}^{j}\partial_{3}^{1-i} W_3$ in $L^2$, respectively, we have
\begin{align}\label{2092650411330n}
(2\mu + \mu')\int|\partial_{\mathrm{h}}^{j}\partial_{3}^{3-i}u_{3}|^2\mathrm{d}x
+2\chi \int \partial_{\mathrm{h}}^{j}\partial_{3}^{1-i}W_{3}\partial_{\mathrm{h}}^{j}\partial_{3}^{3-i}u_3\mathrm{d}x
=\int \partial_{\mathrm{h}}^{i}\partial_{3}^{j-i}\tilde{\mathcal{H}}_{3}\partial_{\mathrm{h}}^{j}\partial_{3}^{3-i}u_3\mathrm{d}x
\end{align}
and
\begin{align}\label{2092650411330}
&\frac{1}{2}\frac{\mathrm{d}}{\mathrm{d}t}\int\bar{\rho}|\partial_{\mathrm{h}}^{j}\partial_{3}^{1-i}W_{3}|^2\mathrm{d}x+4\chi \int |\partial_{\mathrm{h}}^{j}\partial_{3}^{1-i}W_{3}|^2\mathrm{d}x\nonumber\\[1mm]
&=\int \partial_{\mathrm{h}}^{j}\partial_{3}^{1-i}\tilde{\mathcal{G}}_{3}\partial_{\mathrm{h}}^{j}\partial_{3}^{1-i}W_3\mathrm{d}x
-\int\bar{\rho} \partial_{\mathrm{h}}^{j}\partial_{3}^{1-i}(u\cdot\nabla W_3)\partial_{\mathrm{h}}^{j}\partial_{3}^{1-i}W_3\mathrm{d}x.\qquad\qquad
\end{align}
In view of the boundary condition $u\big|_{\partial\Omega}=0$, we find that
\begin{align}\label{2092650411337}
&-\int \bar{\rho}\partial_{\mathrm{h}}^{j}\partial_{3}^{1-i}(u\cdot\nabla W_3)\partial_{\mathrm{h}}^{j}\partial_{3}^{1-i}W_3\mathrm{d}x\nonumber\\[1mm]
&=\frac{1}{2}\int\bar{\rho} \mathrm{div}u|\partial_{\mathrm{h}}^{j}\partial_{3}^{1-i}W_3|^2\mathrm{d}x
-\int\bar{\rho} [\partial_{\mathrm{h}}^{j}\partial_{3}^{1-i},u]\cdot\nabla W_3\partial_{\mathrm{h}}^{j}\partial_{3}^{1-i}W_3\mathrm{d}x.\qquad\qquad
\end{align}
Combining \eqref{2092650411330n}--\eqref{2092650411337}, we conclude that there exists a suitably small constant $\lambda_2> 0$, such that
\begin{align}\label{2092650411338n}
&\frac{1}{2}\frac{\mathrm{d}}{\mathrm{d}t}\int\bar{\rho}|\partial_{\mathrm{h}}^{j}\partial_{3}^{1-i}W_{3}|^2\mathrm{d}x
+\lambda_{2}\big(\|\partial_{\mathrm{h}}^{j}\partial_{3}^{1-i}W_{3}\|_{0}^2+\|\partial_{\mathrm{h}}^{j}\partial_{3}^{3-i}u_3\|_{0}^2\big)
\nonumber\\[1mm]
&\leqslant\bigg|\int \partial_{\mathrm{h}}^{j}\partial_{3}^{1-i}\tilde{\mathcal{H}}_{3}\partial_{\mathrm{h}}^{j}\partial_{3}^{3-i}u_3\mathrm{d}x
+\int \partial_{\mathrm{h}}^{j}\partial_{3}^{1-i}\tilde{\mathcal{G}}_{3}\partial_{\mathrm{h}}^{j}\partial_{3}^{1-i}W_3\mathrm{d}x
+\frac{1}{2}\int\bar{\rho} \mathrm{div}u|\partial_{\mathrm{h}}^{j}\partial_{3}^{1-i}W_3|^2\mathrm{d}x\nonumber\\[1mm]
&\quad-\int\bar{\rho} [\partial_{\mathrm{h}}^{j}\partial_{3}^{1-i},u]\cdot\nabla W_3\partial_{\mathrm{h}}^{j}\partial_{3}^{1-i}W_3\mathrm{d}x\bigg|.
\end{align}

We now turn to estimating the integral terms on the right-hand side of \eqref{2092650411338n}.
By virtue of H\"older's inequality and the product estimate \eqref{product}, we can estimate that
\begin{align*}
\bigg|\frac{1}{2}\int\bar{\rho} \mathrm{div}u|\partial_{\mathrm{h}}^{j}\partial_{3}^{1-i}W_3|^2\mathrm{d}x
-\int\bar{\rho} [\partial_{\mathrm{h}}^{j}\partial_{3}^{1-i},u]\cdot\nabla W_3\partial_{\mathrm{h}}^{j}\partial_{3}^{1-i}W_3\mathrm{d}x\bigg|\lesssim\sqrt{\mathcal{E}}\mathcal{D}.
\end{align*}
Furthermore, invoking H\"older's inequality along with \eqref{1227213738} and \eqref{12281530}, we deduce the bound
\begin{align*}
&\bigg|\int \partial_{\mathrm{h}}^{j}\partial_{3}^{1-i}\tilde{\mathcal{H}}_{3}\partial_{\mathrm{h}}^{j}\partial_{3}^{3-i}u_3\mathrm{d}x
+\int \partial_{\mathrm{h}}^{j}\partial_{3}^{1-i}\tilde{\mathcal{G}}_{3}\partial_{\mathrm{h}}^{j}\partial_{3}^{1-i}W_3\mathrm{d}x\bigg|\\[1mm]
&\lesssim\|\partial_{\mathrm{h}}^{j}\partial_{3}^{1-i}(u_{t},\partial_3q,\nabla\partial_{\mathrm{h}}u)\|_{0}\sqrt{\|\partial_{\mathrm{h}}^{j}\partial_{3}^{1-i}W_{3}\|_{0}^2
+\|\partial_{\mathrm{h}}^{j}\partial_{3}^{3-i}u_3\|_{0}^2}+\sqrt{\mathcal{E}}\mathcal{D}.\quad
\end{align*}
Substituting these estimates into \eqref{2092650411338n} and applying Young's inequality, we arrive at
\begin{align}\label{2092650411339n}
&\frac{1}{2}\frac{\mathrm{d}}{\mathrm{d}t}\int\bar{\rho}|\partial_{\mathrm{h}}^{j}\partial_{3}^{1-i}W_{3}|^2\mathrm{d}x
+\frac{\lambda_{1}}{2}\big(\|\partial_{\mathrm{h}}^{j}\partial_{3}^{1-i}W_{3}\|_{0}^2+\|\partial_{\mathrm{h}}^{j}\partial_{3}^{3-i}u_3\|_{0}^2\big)\qquad\qquad\quad\nonumber\\[1.5mm]
&\lesssim\|\partial_{\mathrm{h}}^{j}\partial_{3}^{1-i}(u_t,\partial_3q,\nabla\partial_{\mathrm{h}}u)\|_{0}^2+\sqrt{\mathcal{E}}\mathcal{D}.
\end{align}

Following the same arguments as for \eqref{2092650411339n}, we apply $\partial_{\mathrm{h}}^{j}\partial_{3}^{1-i}$ to \eqref{2026195637}$_1$ and \eqref{2026195637}$_2$,
take the inner product of the resulting identities with $\partial_{\mathrm{h}}^{j}\partial_{3}^{3-i}u_{\mathrm{h}}$ and $\partial_{\mathrm{h}}^{j}\partial_{3}^{1-i} W_{\mathrm{h}}$ in $L^2$, respectively, and sum the results to obtain
\begin{align}\label{2092650411338}
&\frac{1}{2}\frac{\mathrm{d}}{\mathrm{d}t}\int\bar{\rho}|\partial_{\mathrm{h}}^{j}\partial_{3}^{1-i}W_{\mathrm{h}}|^2\mathrm{d}x\nonumber\\[1mm]
&+(2\mu + \chi)\int|\partial_{\mathrm{h}}^{j}\partial_{3}^{3-i}u_{\mathrm{h}}|^2\mathrm{d}x
+4\chi \int \partial_{\mathrm{h}}^{j}\partial_{3}^{1-i}W_{\mathrm{h}}\cdot\partial_{\mathrm{h}}^{j}\partial_{3}^{3-i}u_{\mathrm{h}}\mathrm{d}x
+4\chi \int |\partial_{\mathrm{h}}^{j}\partial_{3}^{1-i}W_{\mathrm{h}}|^2\mathrm{d}x\nonumber\\[1mm]
&=\int \partial_{\mathrm{h}}^{j}\partial_{3}^{1-i}\tilde{\mathcal{H}}_{\mathrm{h}}\cdot\partial_{\mathrm{h}}^{j}\partial_{3}^{3-i}u_{\mathrm{h}}\mathrm{d}x
+\int \partial_{\mathrm{h}}^{j}\partial_{3}^{1-i}\tilde{\mathcal{G}}_{\mathrm{h}}\cdot\partial_{\mathrm{h}}^{j}\partial_{3}^{1-i}W_{\mathrm{h}}\mathrm{d}x
+\frac{1}{2}\int\bar{\rho} \mathrm{div}u|\partial_{\mathrm{h}}^{j}\partial_{3}^{1-i}W_{\mathrm{h}}|^2\mathrm{d}x\nonumber\\[1mm]
&\quad
-\int\bar{\rho} [\partial_{\mathrm{h}}^{j}\partial_{3}^{1-i},u]\cdot\nabla W_{\mathrm{h}}\cdot\partial_{\mathrm{h}}^{j}\partial_{3}^{1-i}W_{\mathrm{h}}\mathrm{d}x.
\end{align}
Observing the identity
\begin{align*}
&\|\partial_{\mathrm{h}}^{j}\partial_{3}^{3-i}u_{\mathrm{h}}+2 \partial_{\mathrm{h}}^{j}\partial_{3}^{1-i}W_{\mathrm{h}}\|_{0}^2\nonumber\\[1.5mm]
&=
\int|\partial_{\mathrm{h}}^{j}\partial_{3}^{3-i}u_{\mathrm{h}}|^2\mathrm{d}x+4 \int \partial_{\mathrm{h}}^{j}\partial_{3}^{1-i}W_{\mathrm{h}}\cdot\partial_{\mathrm{h}}^{j}\partial_{3}^{3-i}u_{\mathrm{h}}\mathrm{d}x
+4 \int |\partial_{\mathrm{h}}^{j}\partial_{3}^{1-i}W_{\mathrm{h}}|^2\mathrm{d}x
\end{align*}
and the inequality
\begin{align*}
\|\partial_{\mathrm{h}}^{j}\partial_{3}^{1-i}W_{\mathrm{h}}\|_{0}^2\lesssim\|\partial_{\mathrm{h}}^{j}\partial_{3}^{3-i}u_{\mathrm{h}}+2 \partial_{\mathrm{h}}^{j}\partial_{3}^{1-i}W_{\mathrm{h}}\|_{0}^2+\epsilon \|\partial_{\mathrm{h}}^{j}\partial_{3}^{3-i}u_{\mathrm{h}}\|_{0}^2\quad\mbox{for any}\;\epsilon>0.
\end{align*}
Then we can refine \eqref{2092650411338} to be
\begin{align}\label{20926504113339nm}
&\frac{1}{2}\frac{\mathrm{d}}{\mathrm{d}t}\int\bar{\rho}|\partial_{\mathrm{h}}^{j}\partial_{3}^{1-i}W_{\mathrm{h}}|^2\mathrm{d}x
+\lambda_{3}\big(\|\partial_{\mathrm{h}}^{j}\partial_{3}^{1-i}W_{\mathrm{h}}\|_{0}^2+\|\partial_{\mathrm{h}}^{j}\partial_{3}^{3-i}u_{\mathrm{h}}\|_{0}^2\big)
\nonumber\\[1mm]
&\leqslant\bigg|\int \partial_{\mathrm{h}}^{j}\partial_{3}^{1-i}\tilde{\mathcal{H}}_{\mathrm{h}}\cdot\partial_{\mathrm{h}}^{j}\partial_{3}^{3-i}u_{\mathrm{h}}\mathrm{d}x
+\int \partial_{\mathrm{h}}^{j}\partial_{3}^{1-i}\tilde{\mathcal{G}}_{\mathrm{h}}\cdot\partial_{\mathrm{h}}^{j}\partial_{3}^{1-i}W_{\mathrm{h}}\mathrm{d}x
+\frac{1}{2}\int\bar{\rho} \mathrm{div}u|\partial_{\mathrm{h}}^{j}\partial_{3}^{1-i}W_{\mathrm{h}}|^2\mathrm{d}x\nonumber\\[1mm]
&\quad
-\int\bar{\rho} [\partial_{\mathrm{h}}^{j}\partial_{3}^{1-i},u]\cdot\nabla W_{\mathrm{h}}\cdot\partial_{\mathrm{h}}^{j}\partial_{3}^{1-i}W_{\mathrm{h}}\mathrm{d}x\bigg|:=\Pi,
\end{align}
where $\lambda_3$ is a positive constant satisfying $\lambda_2/2>\lambda_3>0$.

Similarly to the derivation of \eqref{2092650411339n}, we have the bound
\begin{align}\label{2092650411539nm}
\Pi\lesssim\|\partial_{\mathrm{h}}^{j}\partial_{3}^{1-i}(u_{t},\nabla\partial_{\mathrm{h}}u)\|_{0}^2+\sqrt{\mathcal{E}}\mathcal{D}
+\bigg|\int \partial_{\mathrm{h}}^{j}\partial_{3}^{1-i}\nabla_{\mathrm{h}}q\cdot\partial_{\mathrm{h}}^{j}\partial_{3}^{3-i}u_{\mathrm{h}}\mathrm{d}x\bigg|.
\end{align}
{{Regarding this integral, in the case where $i=j=0$}}, we see that
\begin{align}\label{202604111539nh}
\bigg|\int \partial_{\mathrm{h}}^{j}\partial_{3}^{1-i}\nabla_{\mathrm{h}}q\cdot\partial_{\mathrm{h}}^{j}\partial_{3}^{3-i}u_{\mathrm{h}}\mathrm{d}x\bigg|
\lesssim\|\partial_{3}\nabla_{\mathrm{h}}q\|_{0}\|\partial_3^{3}u_{\mathrm{h}}\|_{0}.
\end{align}
{{For the case $i=1$ and $j=0$}}, by utilizing the interpolation inequality, we find that
\begin{align}\label{202604111539nm}
\bigg|\int \partial_{\mathrm{h}}^{j}\partial_{3}^{1-i}\nabla_{\mathrm{h}}q\cdot\partial_{\mathrm{h}}^{j}\partial_{3}^{3-i}u_{\mathrm{h}}\mathrm{d}x\bigg|
\lesssim\|\nabla_{\mathrm{h}}q\|_{0}\|\nabla u\|_{1}
\lesssim\|\nabla_{\mathrm{h}}q\|_{0}\|\nabla u\|_{2}^{1/2}\|\nabla u\|_{0}^{1/2}.
\end{align}
{{For the remaining case where $i=j=1$}},
an integration by parts over $\Omega$ yields
\begin{align}\label{2092650411539nnmm}
&\int \partial_{\mathrm{h}}^{j}\partial_{3}^{1-i}\nabla_{\mathrm{h}}q\cdot\partial_{\mathrm{h}}^{j}\partial_{3}^{3-i}u_{\mathrm{h}}\mathrm{d}x
\nonumber\\[1mm]
&=\int \partial_{\mathrm{h}}\partial_3q\partial_{\mathrm{h}}\partial_{3}\mathrm{div}_{\mathrm{h}}u_{\mathrm{h}}\mathrm{d}x
+\int_{\partial\Omega}\nabla_{\mathrm{h}}\partial_{\mathrm{h}}q\cdot\partial_{3}\partial_{\mathrm{h}}u_{\mathrm{h}}\vec{n}_{3}\mathrm{d}x_{\mathrm{h}},
\end{align}
where $\vec{n}_{3}$ denotes the third component of the unit outward normal vector $\vec{n}$ on $\partial\Omega$, namely, $\vec{n}_{3}=1$ on $\mathbb{R}^2\times\{1\}$ and $\vec{n}_{3}=-1$ on $\mathbb{R}^2\times\{0\}$.
Exploiting the dual estimate \eqref{11190840} and the trace estimate \eqref{37190928}, we have the following estimate
\begin{align}\label{20926504111737}
&\bigg|\int_{\partial\Omega}\nabla_{\mathrm{h}}\partial_{\mathrm{h}}q\cdot\partial_{3}\partial_{\mathrm{h}}u_{\mathrm{h}}\vec{n}_{3}\mathrm{d}x_{\mathrm{h}}\bigg|\nonumber\\[1mm]
&\lesssim
|\partial_{\mathrm{h}}q|_{H^{1/2}(\partial\Omega)}|\partial_{3}\partial_{\mathrm{h}}u_{\mathrm{h}}|_{H^{1/2}(\partial\Omega)}
\lesssim\|\partial_{\mathrm{h}}q\|_{1}\|\partial_{3}\partial_{\mathrm{h}}u_{\mathrm{h}}\|_{\underline{1},0}^{1/2}\|\partial_{3}\partial_{\mathrm{h}}u_{\mathrm{h}}\|_{1}^{1/2}.
\end{align}
Inserting \eqref{20926504111737} into \eqref{2092650411539nnmm} leads to
\begin{align}\label{20260411na}
&\bigg|\int \partial_{\mathrm{h}}^{j}\partial_{3}^{1-i}\nabla_{\mathrm{h}}q\cdot\partial_{\mathrm{h}}^{j}\partial_{3}^{3-i}u_{\mathrm{h}}\mathrm{d}x\bigg|\nonumber\\[1mm]
&\lesssim\|\partial_3\partial_{\mathrm{h}}q\|_{0}\|\partial_3\partial_{\mathrm{h}}^2u_{\mathrm{h}}\|_{0}
+\|\partial_{\mathrm{h}}q\|_{1}\|\partial_{3}\partial_{\mathrm{h}}u_{\mathrm{h}}\|_{\underline{1},0}^{1/2}\|\partial_{3}\partial_{\mathrm{h}}u_{\mathrm{h}}\|_{1}^{1/2}
\qquad\nonumber\\[1mm]
&
\lesssim\|\partial_{\mathrm{h}}q\|_{1}\|\partial_{3}\partial_{\mathrm{h}}u_{\mathrm{h}}\|_{\underline{1},0}^{1/2}\|\partial_{3}\partial_{\mathrm{h}}u_{\mathrm{h}}\|_{1}^{1/2}.
\end{align}
Therefore, collecting \eqref{20926504113339nm}--\eqref{202604111539nm} and \eqref{20260411na}, we deduce that
\begin{align}\label{20260411nc}
&\frac{1}{2}\frac{\mathrm{d}}{\mathrm{d}t}\int\bar{\rho}|\partial_{\mathrm{h}}^{j}\partial_{3}^{1-i}W_{\mathrm{h}}|^2\mathrm{d}x
+\lambda_{3}\big(\|\partial_{\mathrm{h}}^{j}\partial_{3}^{1-i}W_{\mathrm{h}}\|_{0}^2+\|\partial_{\mathrm{h}}^{j}\partial_{3}^{3-i}u_{\mathrm{h}}\|_{0}^2\big)\nonumber\\[1mm]
&\lesssim\|\partial_{\mathrm{h}}^{j}\partial_3^{1-i}(u_{t},\nabla\partial_{\mathrm{h}} u)\|_{0}^2
+(1-i)\|\partial_3\partial_{\mathrm{h}}q\|_{0}\|\partial_3^{3}u_{\mathrm{h}}\|_{0}
+\|\nabla q\|_{1}\|u\|_{3}^{1/2}\|u\|_{\underline{2},1}^{1/2}+\sqrt{\mathcal{E}}\mathcal{D}.
\end{align}

Furthermore, by virtue of \eqref{2026195638}$_2$ and \eqref{1227213738}, we directly estimate that
\begin{align}\label{202604151253}
\|\partial_3\partial_{\mathrm{h}}q\|_{0}^2
&\lesssim\bar{\mathcal{D}}+\|\partial_{\mathrm{h}}(\partial_{3}^2u,W)\|_{0}^2+\|(\mathcal{N}^6,\mathcal{N}^7)\|_{1}^2 \nonumber\\[1mm]
&\lesssim\bar{\mathcal{D}}+\|\partial_{\mathrm{h}}(\partial_{3}^2u,W)\|_{0}^2+\sqrt{\mathcal{E}}\mathcal{D}.
\end{align}
Consequently, combining  \eqref{2092650411339n}  with \eqref{20260411nc}, using \eqref{202604151253} and Young's inequality,
we conclude that there exist positive constants $c_{i,j}>0$, such that
\begin{align}\label{20260416nc}
&\frac{\mathrm{d}}{\mathrm{d}t}\sum_{j=0}^{i}c_{i,j}\|\partial_{\mathrm{h}}^{j}\partial_{3}^{1-i}W\|_0^2
+\frac{\lambda_{3}}{2}\sum_{j=0}^{i}c_{i,j}
\big(\|\partial_{\mathrm{h}}^{j}\partial_{3}^{1-i}W\|_{0}^2+\|\partial_{\mathrm{h}}^{j}\partial_{3}^{3-i}u\|_{0}^2\big)\nonumber\\[1mm]
&\lesssim\bar{\mathcal{D}}+  \|(\partial_{3}^{2-i}q,\partial_{3}^{2-i}\partial_{\mathrm{h}} u)\|_{\underline{i},0}^2
+(1-i)\|W\|_{1,0}^2
+\|\nabla q\|_{1}\|u\|_{3}^{1/2}\|u\|_{\underline{2},1}^{1/2}
+\sqrt{\mathcal{E}}\mathcal{D}.
\end{align}

If we define
\begin{align*}%\label{202604162140nc}
\overline{\|\partial_{3}^{1-i}W\|}_{\underline{i},0}^2:=\sum_{j=0}^{i}c_{i,j}\|\partial_{\mathrm{h}}^{j}\partial_{3}^{1-i}W\|_0^2,
\end{align*}
then this functional is equivalent to $\|\partial_{3}^{1-i}W\|_{\underline{i},0}^2$, and we immediately arrive at the desired \eqref{202604102028}.
This completes the proof.
\hfill$\Box$
\end{pf}

\begin{rem}
It should be noted that although we can derive the estimate for $\|\partial_{3}^{2-i}\mathrm{div}u\|_{\underline{i},0}^2$ from Lemma \ref{lem:dfifessimM},
we cannot apply the Stokes estimate to bound $\nabla q$. This is because we do not yet have control over $\nabla\times w$ before Lemma \ref{uwnormal}.
\end{rem}

We now turn to the estimate for $\mathrm{div}w$.
Following the same procedure as in \eqref{202604101338}, by applying the divergence operator to \eqref{0101n}$_3$ and using the identity $\mathrm{div}\nabla\times u=0$,
we obtain
\begin{align}\label{202604101338nm}
\bar{\rho}\big(\mathrm{div}w_{t}+u\cdot\nabla\mathrm{div}w \big)+4\chi\mathrm{div}w=\mathcal{N}^{10},
\end{align}
where $\mathcal{N}^{10}$ is defined in \eqref{20260412a}.
We now establish the following estimates for $\mathrm{div}w$.
\begin{lem}\label{divwnormal}
Under assumption \eqref{aprpioses} with sufficiently small $\delta$, it holds that, %for $0\leqslant i\leqslant 1$,
\begin{align}\label{20260411nd}
&\frac{\mathrm{d}}{\mathrm{d}t}\|\mathrm{div}w\|_{1}^2+c\|\mathrm{div}w\|_{1}^2
\lesssim\sqrt{\mathcal{E}}\mathcal{D}.
\end{align}
\end{lem}
\begin{pf}
Following the derivation of \eqref{202112152140}, we take the $\|\cdot\|_1^2$ on both sides of the identity \eqref{202604101338nm} to obtain
\begin{align}\label{20260411ne}
&4\chi\bar{\rho}\frac{\mathrm{d}}{\mathrm{d}t}\|\mathrm{div}w\|_{1}^2+\|4\chi\mathrm{div}w\|_{1}^2+\|\bar{\rho}\big(\mathrm{div}w_{t}+u\cdot\nabla\mathrm{div}w \big)\|_{1}^2
\lesssim\|u\|_{3}\|w\|_{2}^2+\|\mathcal{N}^{10}\|_{1}^2.
\end{align}
Thanks to \eqref{12281530} and \eqref{aprpioses}, we can estimate that
$$\|\mathcal{N}^{10}\|_{1}^2\lesssim\sqrt{\mathcal{E}}\mathcal{D}.$$
Putting it into \eqref{20260411ne} then yields \eqref{20260411nd}. This completes the proof.
\hfill$\Box$
\end{pf}

Next, we invoke the regularity theory of the elliptic problem to derive higher-order energy estimates for $u$.
We first reformulate the equations \eqref{0101n}$_2$, subject to the boundary condition \eqref{0101c}, as the following Lam\'e problem:
\begin{equation}\label{n0101nn928n}
\begin{cases}
-(\mu+\chi)\Delta u-(\mu+\mu'-\chi)\nabla\mathrm{div}u= F &\mbox{in } \Omega ,\\[0.5mm]
u=0 &\mbox{on } \partial\Omega,
\end{cases}
\end{equation}
where we have define that
\begin{align}
& F :=-\bar{\rho}u_t -  P'(\bar{\rho})\nabla q+2\chi \nabla \times w+\mathcal{N}^6+\mathcal{N}^7.\nonumber
\end{align}
\begin{lem}\label{lem:11292030}
Under assumption \eqref{aprpioses} with sufficiently small $\delta$, it holds that
\begin{align}\label{11292030}
\|u\|_{2}^2&\lesssim\|w\|_{1}^2+\|(u_t,\nabla q)\|_{0}^2.
\end{align}
\end{lem}
\begin{pf}
Applying the Lam\'e estimate \eqref{Ellipticestimate} to \eqref{n0101nn928n}, we have
\begin{align}\label{202108040932}
&\|u\|_{2}^2\lesssim\|F\|_{0}^2
\lesssim \|w\|_{1}^2+\|(u_t,\nabla q)\|_{0}^2+\|(\mathcal{N}^6,\mathcal{N}^7)\|_{0}^2.
\end{align}
Invoking \eqref{1227213738} with $i=0$ and \eqref{aprpioses}, we can estimate that
\begin{align*}
\|(\mathcal{N}^6,\mathcal{N}^7)\|_{0}^2&\lesssim\|(u,q)\|_{2}^2\big(\|u\|_{2}^2+\|(\nabla w,\nabla q)\|_{0}^2\big)\nonumber\\[1mm]
&\lesssim\delta\big(\|u\|_{2}^2+\|w\|_{1}^2+\|\nabla q\|_{0}^2\big).
\end{align*}
Substituting the above estimate into \eqref{202108040932}, we immediately obtain \eqref{11292030} for sufficiently small $\delta$.
This completes the proof.
\hfill$\Box$
\end{pf}

%%%%%%%%%%%%%%%%%%%%%%%%%%%%%%%%%%%%%%%%%%%%%%%%%%%%%%%%%%%%%%%%%%%%%%%
%\subsection{Energy inequality}
\subsection{A priori energy estimate}
In this subsection, we assemble the aforementioned estimates to establish the \emph{a priori} energy estimates for $\mathcal{E}$ and $\mathcal{D}$.
We begin by proving the equivalence estimate of $\mathcal{E}$.

%%%%%%%%%%%%%%%%%%%%%%%%%%%%%%%%%%%%%%%%%%%%%%%%%%%%%%%%%%%%%%%%%%%%%%%
\begin{lem}\label{lem:12281620}
Under assumption \eqref{aprpioses} with sufficiently small $\delta$, it holds that
\begin{align}\label{12281622}
\mathcal{E}\;\;\mbox{is equivalent to}\;\;\|(q, u, w)\|_{2}^2.
\end{align}
\end{lem}
\begin{pf}
To obtain \eqref{12281622}, it suffices to verify
\begin{align}
\label{201807311651}
\|(q_{t},w_{t})\|_1^2+\|u_t\|_0^2\lesssim \|\nabla q\|_{1}^2+\|(u, w)\|_{2}^2.
\end{align}
Indeed, a direct application of \eqref{0101nb}$_1$--\eqref{0101nb}$_3$ leads to
\begin{align*}
\|(q_{t},w_{t})\|_1^2+\|u_t\|_0^2\lesssim \|u\|_{2}^2+\| w\|_{1}^2+\|\nabla q\|_{0}^2+\|(\mathcal{N}^5,\mathcal{N}^8)\|_{1}^2+\|(\mathcal{N}^6,\mathcal{N}^7)\|_{0}^2.
\end{align*}
Applying \eqref{1227213738}--\eqref{1227213738n}, we arrive at \eqref{201807311651} for sufficiently small $\delta$.
This completes the proof.
\hfill$\Box$
\end{pf}

With Proposition \ref{pro12281500} and Lemmas \ref{lem:dfifessimM}--\ref{lem:12281620} in hand, we derive the following energy inequality.
\begin{pro}
\label{pro:0501n}
Under assumption \eqref{aprpioses} with sufficiently small $\delta$,
there exists a function $\tilde{\mathcal{E}}$, which is equivalent to $\mathcal{E}$, such that
\begin{align}   \label{emdslds}
\frac{\mathrm{d}}{\mathrm{d}t} \tilde{\mathcal{E}}+c\mathcal{D}\leqslant0.
\end{align}
\end{pro}
\begin{pf}
To begin with, recalling that $W:=\nabla\times w$ and noting that $\mathrm{div}W=0$, we have
\begin{align}   \label{01013630}
\partial_3W_3=-\mathrm{div}_{\mathrm{h}}W_{\mathrm{h}}.
\end{align}
Furthermore, since $W_3=\partial_1w_2-\partial_2w_1$, we obtain the estimate
\begin{align}   \label{01011630}
\|W_3\|_{\underline{i},0}^2\lesssim\|w\|_{\underline{2},0}\quad\mbox{for}\;0\leqslant i\leqslant 1.
\end{align}
Employing \eqref{01013630}--\eqref{01011630}, we deduce from \eqref{202112152140} and \eqref{202604102028} that
there exist a constant $\lambda_4> 0$ and three suitably large constants $c_1$,  $c_2$  and $c_3$ such that
\begin{align}   \label{01011430}
\frac{\mathrm{d}}{\mathrm{d}t} \mathcal{E}_1+\lambda_4\mathcal{D}_1\lesssim
\bar{\mathcal{D}}+\|\nabla q\|_{1}\|u\|_{3}^{1/2}\|u\|_{\underline{2},1}^{1/2}+\sqrt{\mathcal{E}}\mathcal{D},
\end{align}
where the functionals $\mathcal{E}_1$ and $\mathcal{D}_1$ are defined by
\begin{align*}
&\mathcal{E}_1:=\overline{\|\partial_{3}W\|}_{0}^2+c_1\|\partial_{3}^{2}q\|_{0}^2+c_3\big(\overline{\|W\|}_{\underline{1},0}^2+c_2\|\partial_{3}q\|_{\underline{1},0}^2\big) ,\\[1mm]
&\mathcal{D}_1:=\big(\|\partial_{3}W\|_{0}^2+\|\partial_{3}^{3}u\|_0^2\big)+c_1\|\partial_{3}^{2}q\|_{0}^2
+c_3\big((\|W\|_{\underline{1},0}^2+\|\partial_{3}^{2}u\|_{\underline{1},0}^2)+c_2\|\partial_{3}q\|_{\underline{1},0}^2\big).
\end{align*}
Summing up \eqref{12281503nmnm}, \eqref{20260411nd} and \eqref{01011430}, we conclude that for a suitably small constant $\lambda_5> 0$
and a suitably large constant $c_4$, the following inequality holds:
\begin{align}   \label{010117nm30}
\frac{\mathrm{d}}{\mathrm{d}t} \mathcal{E}_2+\lambda_5\mathcal{D}_2\lesssim
\|\nabla q\|_{1}\|u\|_{3}^{1/2}\|u\|_{\underline{2},1}^{1/2}+\sqrt{\mathcal{E}}\mathcal{D},
\end{align}
with $\mathcal{E}_2$ and $\mathcal{D}_2$ defined by
\begin{align*}
&\mathcal{E}_2:=\|\mathrm{div}w\|_{1}^2+\mathcal{E}_1+c_4\tilde{\bar{\mathcal{E}}}\quad\mbox{and}\quad
\mathcal{D}_2:=\|\mathrm{div}w\|_{1}^2+\mathcal{D}_1+c_4\bar{\mathcal{D}}.
\end{align*}

In view of \eqref{0101nb}$_2$, \eqref{1227213738} and the assumption \eqref{aprpioses} with sufficiently small $\delta$, we have
\begin{align*}
&%\label{010136n30}
\|\nabla q\|_{1}\lesssim\big(\|u\|_{3}+\|w\|_{2}+\bar{\mathcal{D}}\big).
\end{align*}
In addition, straightforward computations show that
\begin{align*}
&%\label{01013530}
\|u\|_{3}\lesssim\|u\|_{\underline{2},1}+\sum_{i=0}^1\|\partial_3^{3-i}u\|_{\underline{i},0},\\
&%\label{01013630}
\|w\|_{2}\lesssim\|w\|_{\underline{2},0}+\|W\|_{1}+\|\mathrm{div}w\|_{1}.
\end{align*}
Therefore, using \eqref{11292030}, \eqref{12281622}--\eqref{201807311651} and the above three estimates, we conclude that
for sufficiently small $\delta$, the following equivalence relations hold:
\begin{align}
&\label{01011507}
\mathcal{E}_2,\;\mathcal{E}\;\mbox{and}\;\|(q, u, w)\|^2_2
\;\;\mbox{are equivalent},\\[1mm]
&\label{01011455}
\mathcal{D}_2\;\;\mbox{is equivalent to}\;\;\mathcal{D}.
\end{align}
Consequently, combining \eqref{12281503nmnm} with \eqref{010117nm30} and applying Young's inequality alongside \eqref{01011507}--\eqref{01011455}, we deduce that there exist a suitably small constant $\lambda_6> 0$ and a suitably large constant $c_5$ such that
\begin{align}   \label{010117nml30}
%\frac{\mathrm{d}}{\mathrm{d}t} \tilde{\mathcal{E}}+c\mathcal{D}\leqslant
\frac{\mathrm{d}}{\mathrm{d}t} \tilde{\mathcal{E}}+\lambda_6(\mathcal{D}_2+c_5\bar{\mathcal{D}})\leqslant0,
\end{align}
where $\tilde{\mathcal{E}}$ is defined by
\begin{align*}
&\tilde{\mathcal{E}}:=\mathcal{E}_2+c_5\tilde{\bar{\mathcal{E}}}.
\end{align*}
Clearly, $\tilde{\mathcal{E}}$ is also equivalent to $\mathcal{E}$ in light of \eqref{01011507}.
Finally, the desired energy inequality \eqref{emdslds} is a direct consequence of \eqref{01011455}--\eqref{010117nml30}.
This completes the proof.
\hfill$\Box$
\end{pf}

%%%%%%%%%%%%%%%%%%%%%%%%%%%%%%%%%%%%%%%%%%%%%%%%%%%%%%%%%%%%%%%%%%%%
%\subsection{A priori estimate}
%We can now arrive at the following \emph{a priori} energy estimate, which immediately closes the \emph{a priori} assumption \eqref{aprpioses}.
%Hence, by combining this estimate with the local existence result and a standard continuity argument, we establish Theorem \ref{thm1}.
We can now establish the following \emph{a priori} energy estimate, which combines with the well-posedness result and a continuity argument,
yields Theorem \ref{thm1}.
\begin{thm}\label{thm1nm}
There exists a sufficiently small constant $\delta_1\in(0,1)$ such that, if \eqref{aprpioses} holds for $\delta\leqslant\delta_1$, then
\begin{align}   \label{emdsldsa37n}
\mathcal{E}(t)+\int_0^{t}\mathcal{D}(\tau)\mathrm{d}\tau\lesssim \|(q^0,u^0,w^0)\|_{2}^2
\quad\;\mathrm{for \;all}\;\;t\in(0,T].
\end{align}
\end{thm}
\begin{pf}
Integrating the above inequality \eqref{010117nml30} over $(0,t)$ and applying \eqref{12281622} at $t=0$,
we then arrive at the estimate \eqref{emdsldsa37n}.
\hfill$\Box$
\end{pf}

%%%%%%%%%%%%%%%%%%%%%%%%%%%%%%%%%%%%%%%%%%%%%%%%%%%%%%%%%%%%%%%%%%%%%%%%%%%%%%%%%%%%%%%%%%%%%%%%%%%%%%%%%%%%%%%%55
\section{Proof of Theorem \ref{thm2}} \label{2025thm2}
This section is dedicated to the proof of Theorem \ref{thm2}.
The main objective is also to derive the \emph{a priori} energy estimates for the solution of the problem \eqref{0102n}--\eqref{0102c}.
To this end, let $(q,u,w)$ be a solution to \eqref{0102n}--\eqref{0102c} with an associated pressure $p$ defined on $\Omega\times [0,T]$ with $T>0$
and assume that
\begin{align}
&\label{aprpioseses}
\sup_{t\in[0,T]}\left\|(q, u,w)(t)\right\|_{2}< \delta\in(0,1),
\end{align}
where $\delta$ is sufficiently small.
Thanks to \eqref{aprpioseses} and the Sobolev embedding inequality \eqref{embed2}, it is straightforward to deduce that
\begin{align}\label{202604121856nhnm}
\bar{\rho}/2\leqslant \inf_{x\in\Omega}\{\bar{\rho}+q\}\leqslant\sup_{x\in\Omega}\{\bar{\rho}+q\}\leqslant 3\bar{\rho}/2.
\end{align}

%%%%%%%%%%%%%%%%%%%%%%%%%%%%%%%%%%%%%%%%%%%%%%%%%%%%%%%%%
\subsection{Preliminary estimates}
Analogous to the compressible case, in order to exploit the nonlinear structure of \eqref{0102n} to derive tangential energy estimates, we apply
$\partial^{\alpha}$ with $\alpha\in\mathbb{N}^{1+2}$ to \eqref{0102n}--\eqref{0102c} to obtain
\begin{equation}\label{202604120102ne}
\begin{cases}
(\bar{\rho}+q)(\partial^{\alpha}u_t + u \cdot \nabla \partial^{\alpha}u)
+\nabla \partial^{\alpha}q - (\mu + \chi) \Delta \partial^{\alpha}u
=2\chi \nabla \times \partial^{\alpha}w+\mathcal{M}^{1,\alpha}
 \quad &\mbox{in } \Omega, \\[0.5mm]
(\bar{\rho}+q)\big(\partial^{\alpha}w_t + u \cdot \nabla \partial^{\alpha}w\big)  + 4\chi \partial^{\alpha} w = 2\chi \nabla \times \partial^{\alpha}u
+\mathcal{M}^{2,\alpha} \quad &\mbox{in } \Omega,\\[0.5mm]
\mathrm{div}\partial^{\alpha}u=0 \quad &\mbox{in } \Omega,\\[0.5mm]
\partial^{\alpha}u=0 &\mbox{on }\partial\Omega,
\end{cases}
\end{equation}
where $\mathcal{M}^{1,\alpha}$ and $\mathcal{M}^{2,\alpha}$ are defined by
\begin{align*}
&\mathcal{M}^{1,\alpha}:=-(\bar{\rho}+q)[\partial^{\alpha},u ]\cdot \nabla u-[\partial^{\alpha},q](u_t + u \cdot \nabla u) ,\\[1.5mm]
&\mathcal{M}^{2,\alpha}:=-(\bar{\rho}+q)[\partial^{\alpha},u ]\cdot \nabla w-[\partial^{\alpha},q](w_t+u\cdot \nabla w).
\end{align*}

Similarly to \eqref{0101nb}, when utilizing the linear structure of \eqref{0101n}, it is useful to write it as a perturbation of the linearized system:
\begin{equation}\label{0102nb}
\begin{cases}
\bar{\rho}u_t +  \nabla q - (\mu + \chi) \Delta u-2\chi \nabla \times w=\mathcal{M}^3
 \qquad &\mbox{in } \Omega, \\[0.5mm]
\bar{\rho}w_t  + 4\chi w- 2\chi \nabla \times u=\mathcal{M}^4 \quad &\mbox{in } \Omega,\\[0.5mm]
\mathrm{div}u=0 \quad &\mbox{in } \Omega,\\[0.5mm]
u=0 \quad &\mbox{on } \partial\Omega,
\end{cases}
\end{equation}
where $\mathcal{M}^3$ and $\mathcal{M}^4$ are given by
\begin{align*}
&\mathcal{M}^3:=
-\bar{\rho} u \cdot \nabla u-q(u_t + u \cdot \nabla u),\\[1.5mm]
&\mathcal{M}^4:=-\bar{\rho}u \cdot \nabla w-q(w_t + u \cdot \nabla w).
\end{align*}

We now establish estimates for these nonlinear terms.
\begin{lem}
\label{202606291049}
Under assumption \eqref{aprpioseses} with sufficiently small $\delta$, the following estimates hold:
\begin{align}
\label{12272120es12}
& \|\mathcal{M}^{1,\alpha}\|_0\lesssim\|(q,u)\|_{2}\sqrt{\mathfrak{D}}
\quad\quad\;\mathrm{for}\;\;|\alpha|\leqslant1\;\mathrm{or}\;\alpha_0=1,\\[1mm]
\label{12272120es}
& \|\mathcal{M}^{2,\alpha}\|_0\lesssim\|(q,u)\|_{2}\sqrt{\mathfrak{D}}
\quad\quad\;\mathrm{for}\;\;|\alpha|\leqslant2,\\[1mm]
&\label{1227213738es}
 \|\mathcal{M}^{3}\|_{i}\lesssim\|(q,u)\|_{2}\big(\|u\|_{1+i}+\|u_t\|_{i}\big)
\quad\;\mathrm{for}\;\;0\leqslant i\leqslant1,\\[1mm]
& \label{12276213738es}
\|\mathcal{M}^{4}\|_{i}\lesssim\|(q,u)\|_{2}\big(\|w\|_{1+i}+\|w_{t}\|_{i}\big)
\quad\mathrm{for}\;\;0\leqslant i\leqslant1.
\end{align}

In addition, let $\mathcal{M}^{5}$ and $\mathcal{M}^6$ be defined by
\begin{align}
&\label{20260412aes}
\mathcal{M}^{5}:=-\nabla\times\big(q( w_{t}+u\cdot\nabla w)\big)-\bar{\rho}[\nabla\times, u]\cdot\nabla w,\\[1mm]
&\label{20260412bes}
\mathcal{M}^{6}:=-\mathrm{div}\big(q( w_{t}+u\cdot\nabla w)\big)-\bar{\rho}[\mathrm{div}, u]\cdot\nabla w,
\end{align}
then there holds that
\begin{align}
& \|\mathcal{M}^{5}\|_1+\|\mathcal{M}^{6}\|_1\lesssim \|(q,u)\|_{2}\sqrt{\mathfrak{D}}. \label{12281530es}
\end{align}
\end{lem}
\begin{pf}
By virtue of \eqref{0102n}$_1$, we have
\begin{align*}
\|q_t\|_{i}\lesssim\|u\cdot\nabla q\|_{i}\lesssim\|u\|_{2}\| q\|_{1+i}\quad\mbox{for}\;\;0\leqslant i\leqslant1.
\end{align*}
Since the proof of Lemma \ref{202606291049} follows similarly to that of Lemma \ref{201806291049}, we omit the details.
\hfill$\Box$
\end{pf}

%%%%%%%%%%%%%%%%%%%%%%%%%%%%%%%%%%%%%%%%%%%%%%%%%%%%%%%%%%%%%%%%%%%%%%%%%
\subsection{Energy evolution for $(q, u, w)$}
In this subsection, we first derive tangential energy estimates involving the temporal and horizontal spatial derivatives of the solution. We then employ the regularity theory for the Stokes problem to obtain further estimates for $(u,p)$, and finally exploit the ODE structure of the linear perturbed system to deduce the normal estimates for $w$.

We define the tangential energy by
\begin{align} \label{20260412nhes}
\bar{\mathfrak{E}}:=\sum_{j=0}^{1}\|\partial_t^j(u,w)\|_{\underline{2-j},0}^2
\end{align}
and the corresponding tangential dissipation by
\begin{align} \label{20260412nhnes}
\bar{\mathfrak{D}}:=\sum_{j=0}^{1}\left(\|\partial_t^ju\|_{\underline{2-j},1}^2+\|\partial_t^jw\|_{\underline{2-j},0}^2\right).
\end{align}
We then have the following tangential energy estimate.
\begin{pro}\label{pro12281500es}
Under assumption \eqref{aprpioseses} with sufficiently small $\delta$,
there exists a function $\tilde{\bar{\mathfrak{E}}}$, which is equivalent to $\bar{\mathfrak{E}}$, such that
\begin{align} \label{12281503nmnmes}
&\frac{\mathrm{d}}{\mathrm{d}t}\tilde{\bar{\mathfrak{E}}}+c\bar{\mathfrak{D}}\lesssim \|(q,u)\|_{2}\mathfrak{D}.
\end{align}
\end{pro}
\begin{pf}
The proof of Proposition \ref{pro12281500es} is similar to that of Proposition \ref{pro12281500}, so we omit it here.
\hfill$\Box$
\end{pf}

Next, we invoke the regularity theory of the Stokes problem to derive further estimates of $(u,p)$.
Indeed, it follows from \eqref{0102n}--\eqref{0102c} that $(u,p)$ satisfies the following Stokes problem:
\begin{equation}\label{n0101nn928nes}
\begin{cases}
-(\mu+\chi)\Delta u+\nabla p= G\quad &\mbox{in } \Omega ,\\[0.5mm]
\mathrm{div} u=0 &\mbox{in } \Omega ,\\[0.5mm]
u=0 &\mbox{on } \partial\Omega,
\end{cases}
\end{equation}
where $G$ is given by
\begin{align}
& G:=-\bar{\rho}u_t +2\chi \nabla \times w+\mathcal{M}^3.\nonumber
\end{align}
\begin{lem}\label{lem:11292030es}
Under assumption \eqref{aprpioseses} with sufficiently small $\delta$, there holds that
\begin{align}
&\label{11292030es}
\|u\|_{\underline{i},2}^2+\|\nabla p\|_{\underline{i},0}^2\lesssim\|\nabla\times w\|_{\underline{i},0}^2+\|u_t\|_{i}^2+\|u\|_{\underline{i},0}^2\quad\mathrm{for}\;\;0\leqslant i\leqslant1,\\[1mm]
&\label{12372030es}
\|u\|_{3}^2+\|\nabla p\|_{1}^2
\lesssim\|\nabla\times w\|_{1}^2+\|u_t\|_{1}^2+\|u\|_{0}^2.
\end{align}
\end{lem}
\begin{pf}
Applying the Stokes estimate \eqref{11190928} to \eqref{n0101nn928nes} and using \eqref{1227213738es}, we have
\begin{align}\label{11293730es}
\|u\|_{2+i}^2+\|\nabla p\|_{i}^2&\lesssim\|u\|_{0}^2+\|G\|_{i}^2\nonumber\\[0.5mm]
&\lesssim\|u\|_{0}^2+\|(u_t,\nabla\times w)\|_{i}^2+\|(q,u)\|_{2}^2\big(\|u\|_{1+i}^2+\|u_t\|_{i}^2\big).
\end{align}

Similarly, the pair $(\partial_{\mathrm{h}}u,\partial_{\mathrm{h}}p)$ solves the following Stokes problem
\begin{equation*}%\label{n0101nn928nesbnm}
\begin{cases}
-(\mu+\chi)\Delta \partial_{\mathrm{h}}u+\nabla \partial_{\mathrm{h}}p= \partial_{\mathrm{h}}G\quad &\mbox{in } \Omega ,\\[0.5mm]
\mathrm{div} \partial_{\mathrm{h}}u=0 &\mbox{in } \Omega ,\\[0.5mm]
\partial_{\mathrm{h}}u=0 &\mbox{on } \partial\Omega,
\end{cases}
\end{equation*}
and enjoys the following estimate
\begin{align}\label{11293730ens}
\|\partial_{\mathrm{h}}u\|_{2}^2+\|\nabla \partial_{\mathrm{h}}p\|_{0}^2&\lesssim \|\partial_{\mathrm{h}}u\|_{0}^2 + \|\partial_{\mathrm{h}}G\|_{0}^2\nonumber\\[0.5mm]
&
\lesssim\|(u,u_t)\|_{1,0}^2+\|\nabla\times w\|_{1,0}^2+\|(q,u)\|_{2}^2\big(\|u\|_{2}^2+\|u_t\|_{1}^2\big).
\end{align}
Consequently, combining \eqref{11293730es}--\eqref{11293730ens} then leads to \eqref{11292030es}--\eqref{12372030es} for sufficiently small $\delta$.
\hfill$\Box$
\end{pf}

Next, we utilize the ODE structure of the linear perturbed system to recover the estimates on the normal derivatives of $w$.
We continue to denote
$$W=(W_1,W_2,W_3)^{\top}:=\nabla\times w.$$
Similarly to the derivation of \eqref{2026195637}--\eqref{2026195638}, and recalling the divergence-free condition $\mathrm{div}u=0$, we observe that
$(\partial_3^2 u_{\mathrm{h}},W_{\mathrm{h}})$ satisfies
\begin{equation}\label{2026195637es}
\begin{cases}
(\mu + \chi)\partial_{3}^2u_{\mathrm{h}}+2\chi W_{\mathrm{h}}
=\bar{\rho}\partial_tu_{\mathrm{h}}  +  \nabla_{\mathrm{h}} p- (\mu + \chi) \Delta_{\mathrm{h}} u_{\mathrm{h}}
-\mathcal{M}^3_{\mathrm{h}}:=\tilde{\mathcal{M}}_{\mathrm{h}}
 &\mbox{in } \Omega ,\\[1mm]
\bar{\rho}\big(\partial_{t}W_{\mathrm{h}}+u\cdot\nabla W_{\mathrm{h}} \big)+4\chi W_{\mathrm{h}}+2\chi\partial_{3}^2u_{\mathrm{h}}
=-2\chi\Delta_{\mathrm{h}} u_{\mathrm{h}}+\mathcal{M}^5_{\mathrm{h}}:=\tilde{\mathfrak{M}}_{\mathrm{h}} &\mbox{in } \Omega,
\end{cases}
\end{equation}
and that $W_{3}$ satisfies
\begin{equation}\label{2026195638es}
%\begin{cases}
%(2\mu + \chi)\partial_{3}^2u_{3}+2\chi W_{3}
%=\bar{\rho}\partial_tu_{3}  + {\color{red} \partial_{3} p} - (\mu + \chi) \Delta_{\mathrm{h}} u_{3} + \mathcal{M}^3_{3}:=\tilde{\mathcal{M}}_{3}
% &\mbox{in } \Omega ,\\[1mm]
\bar{\rho}\big(\partial_{t}W_{3}+u\cdot\nabla W_{3} \big)+4\chi W_{3}
=2\chi\big(\partial_3\mathrm{div}_{\mathrm{h}}u_{\mathrm{h}}-\Delta_{\mathrm{h}} u_{3}\big)+\mathcal{M}^5_{3}:=\tilde{\mathfrak{M}}_{3}\quad \mbox{in } \Omega.
%\end{cases}
\end{equation}

Following the same arguments as those for \eqref{202604101338nm}, we can see that $\mathrm{div}w$ satisfies
\begin{align}\label{202604101338nmes}
\bar{\rho}\big(\mathrm{div}w_{t}+u\cdot\nabla\mathrm{div}w \big)+4\chi\mathrm{div}w=\mathcal{M}^{6},
\end{align}
where $\mathcal{M}^{6}$ is defined in \eqref{20260412bes}.
We now present the following estimates for $W$ and $\mathrm{div}w$.
\begin{lem}\label{uwnormalgh}
Under assumption \eqref{aprpioseses} with sufficiently small $\delta$, there holds that, %for $0\leqslant i\leqslant 1$,
\begin{align}\label{202604102028gh}
&\frac{\mathrm{d}}{\mathrm{d}t}\widetilde{\|W\|}_{1}^2+c\|W\|_{1}^2
\lesssim \bar{\mathfrak{D}} +\|\nabla p\|_{1}\|u\|_{3}^{1/2}\|u\|_{\underline{2},1}^{1/2}
+\|(q,u,w)\|_{2}\mathfrak{D},\\[1mm]
&\label{20260411ndkl}
\frac{\mathrm{d}}{\mathrm{d}t}\|\mathrm{div}w\|_{1}^2+c\|\mathrm{div}w\|_{1}^2
\lesssim\|(q,u)\|_{2}\mathfrak{D},
\end{align}
where $\widetilde{\|W\|}_{1}^2$ is equivalent to $\|W\|_{1}^2$.
\end{lem}
\begin{pf}
For $0\leqslant j\leqslant i\leqslant 1$, applying $\partial_{\mathrm{h}}^{j}\partial_{3}^{1-i}$  to \eqref{2026195637es}$_1$ and \eqref{2026195637es}$_2$,
taking the inner product of the resulting identities with $\partial_{\mathrm{h}}^{j}\partial_{3}^{3-i}u_{\mathrm{h}}$ and $\partial_{\mathrm{h}}^{j}\partial_{3}^{1-i} W_{\mathrm{h}}$ in $L^2$, respectively,
integrating by parts with the aid of \eqref{0102nb}$_3$--\eqref{0102nb}$_4$, we arrive at
\begin{align}\label{2092650411338fg}
&\frac{1}{2}\frac{\mathrm{d}}{\mathrm{d}t}\int\bar{\rho}|\partial_{\mathrm{h}}^{j}\partial_{3}^{1-i}W_{\mathrm{h}}|^2\mathrm{d}x\nonumber\\[1mm]
&+(\mu + \chi)\int|\partial_{\mathrm{h}}^{j}\partial_{3}^{3-i}u_{\mathrm{h}}|^2\mathrm{d}x
+4\chi \int \partial_{\mathrm{h}}^{j}\partial_{3}^{1-i}W_{\mathrm{h}}\cdot\partial_{\mathrm{h}}^{j}\partial_{3}^{3-i}u_{\mathrm{h}}\mathrm{d}x
+4\chi \int |\partial_{\mathrm{h}}^{j}\partial_{3}^{1-i}W_{\mathrm{h}}|^2\mathrm{d}x\nonumber\\[1mm]
&=\int \partial_{\mathrm{h}}^{j}\partial_{3}^{1-i}\tilde{\mathcal{M}}_{\mathrm{h}}\cdot\partial_{\mathrm{h}}^{j}\partial_{3}^{3-i}u_{\mathrm{h}}\mathrm{d}x
+\int \partial_{\mathrm{h}}^{j}\partial_{3}^{1-i}\tilde{\mathfrak{M}}_{\mathrm{h}}\cdot\partial_{\mathrm{h}}^{j}\partial_{3}^{1-i}W_{\mathrm{h}}\mathrm{d}x\nonumber\\[1mm]
&\quad
-\int\bar{\rho} [\partial_{\mathrm{h}}^{j}\partial_{3}^{1-i},u]\cdot\nabla W_{\mathrm{h}}\cdot\partial_{\mathrm{h}}^{j}\partial_{3}^{1-i}W_{\mathrm{h}}\mathrm{d}x.
\end{align}
Following the same lines as the derivation of \eqref{20926504113339nm}, we can refine \eqref{2092650411338fg} into the form
\begin{align}\label{20926504113339nm12fg}
&\frac{1}{2}\frac{\mathrm{d}}{\mathrm{d}t}\int\bar{\rho}|\partial_{\mathrm{h}}^{j}\partial_{3}^{1-i}W_{\mathrm{h}}|^2\mathrm{d}x
+\lambda_{7}\big(\|\partial_{\mathrm{h}}^{j}\partial_{3}^{1-i}W_{\mathrm{h}}\|_{0}^2+\|\partial_{\mathrm{h}}^{j}\partial_{3}^{3-i}u_{\mathrm{h}}\|_{0}^2\big)
\nonumber\\[1mm]
&\leqslant\bigg|
\int \partial_{\mathrm{h}}^{j}\partial_{3}^{1-i}\tilde{\mathcal{M}}_{\mathrm{h}}\cdot\partial_{\mathrm{h}}^{j}\partial_{3}^{3-i}u_{\mathrm{h}}\mathrm{d}x
+\int \partial_{\mathrm{h}}^{j}\partial_{3}^{1-i}\tilde{\mathfrak{M}}_{\mathrm{h}}\cdot\partial_{\mathrm{h}}^{j}\partial_{3}^{1-i}W_{\mathrm{h}}\mathrm{d}x\qquad\qquad\quad\nonumber\\[1mm]
&\quad\;
-\int\bar{\rho} [\partial_{\mathrm{h}}^{j}\partial_{3}^{1-i},u]\cdot\nabla W_{\mathrm{h}}\cdot\partial_{\mathrm{h}}^{j}\partial_{3}^{1-i}W_{\mathrm{h}}\mathrm{d}x\bigg|:=\Xi
\end{align}
holds for some positive constant $\lambda_7$.

Employing H\"older's and Young's inequalities, the product estimate \eqref{product}, and \eqref{1227213738es}--\eqref{12276213738es}, we readily obtain the bound:
\begin{align}\label{2092650411539nmfg}
\Xi\lesssim\|\partial_{\mathrm{h}}^{j}\partial_{3}^{1-i}(u_{t},\partial_{\mathrm{h}}^2u)\|_{0}^2+\|(q,u,w)\|_{2}\mathfrak{D}
+\bigg|\int \partial_{\mathrm{h}}^{j}\partial_{3}^{1-i}\nabla_{\mathrm{h}}p\cdot\partial_{\mathrm{h}}^{j}\partial_{3}^{3-i}u_{\mathrm{h}}\mathrm{d}x\bigg|.
\end{align}
For the pressure term, adopting the same arguments as in \eqref{202604111539nm}--\eqref{20260411na} yields
\begin{align}\label{20926537nmfg}
\bigg|\int \partial_{\mathrm{h}}^{j}\partial_{3}^{1-i}\nabla_{\mathrm{h}}p\cdot\partial_{\mathrm{h}}^{i}\partial_{3}^{3-i}u_{\mathrm{h}}\mathrm{d}x\bigg|
\lesssim(1-i)\|\partial_3\partial_{\mathrm{h}}p\|_{0}\|\partial_3^{3}u_{\mathrm{h}}\|_{0}
+\|\nabla p\|_{1}\|u\|_{3}^{1/2}\|u\|_{\underline{2},1}^{1/2}.
\end{align}
Substituting \eqref{2092650411539nmfg}--\eqref{20926537nmfg} into \eqref{20926504113339nm12fg} thus yields
\begin{align}\label{20260415nh}
&\frac{1}{2}\frac{\mathrm{d}}{\mathrm{d}t}\int\bar{\rho}|\partial_{\mathrm{h}}^{j}\partial_{3}^{1-i}W_{\mathrm{h}}|^2\mathrm{d}x
+\lambda_{6}\big(\|\partial_{\mathrm{h}}^{j}\partial_{3}^{1-i}W_{\mathrm{h}}\|_{0}^2+\|\partial_{\mathrm{h}}^{j}\partial_{3}^{3-i}u_{\mathrm{h}}\|_{0}^2\big)\nonumber\\[1mm]
&\lesssim\|\partial_3^{1-i}(u_{t},\partial_{\mathrm{h}}^2 u)\|_{\underline{i},0}^2
+(1-i)\|\partial_3\partial_{\mathrm{h}}p\|_{0}\|\partial_3^{3}u_{\mathrm{h}}\|_{0}
+\|\nabla p\|_{1}\|u\|_{3}^{1/2}\|u\|_{\underline{2},1}^{1/2}
+\|(q,u,w)\|_{2}\mathfrak{D}.
\end{align}

Similarly, applying $\partial_{\mathrm{h}}^{j}\partial_{3}^{1-i}$ to \eqref{2026195638es}, taking the inner product of the resulting identity with $\partial_{\mathrm{h}}^{j}\partial_{3}^{1-i} W_{3}$ in $L^2$,
integrating by parts, and using \eqref{0102nb}$_3$--\eqref{0102nb}$_4$, we find that
\begin{align*}%\label{20926504153339nmfg}
&\frac{1}{2}\frac{\mathrm{d}}{\mathrm{d}t}\int\bar{\rho}|\partial_{\mathrm{h}}^{j}\partial_{3}^{1-i}W_{3}|^2\mathrm{d}x+4\chi\|\partial_{\mathrm{h}}^{j}\partial_{3}^{1-i}W_{3}\|_{0}^2
\nonumber\\[1mm]
&\leqslant\bigg|
\int \partial_{\mathrm{h}}^{j}\partial_{3}^{1-i}\tilde{\mathfrak{M}}_{3}\cdot\partial_{\mathrm{h}}^{j}\partial_{3}^{1-i}W_{3}\mathrm{d}x
-\int\bar{\rho} [\partial_{\mathrm{h}}^{j}\partial_{3}^{1-i},u]\cdot\nabla W_{3}\cdot\partial_{\mathrm{h}}^{j}\partial_{3}^{1-i}W_{3}\mathrm{d}x\bigg|:=\Upsilon.
\end{align*}
Moreover, the term $\Upsilon$ admits the following bound:
\begin{align*}%\label{20260415nmfg}
\Upsilon\lesssim\|\partial_{\mathrm{h}}^{j}\partial_{3}^{1-i}(u_{t},\nabla\partial_{\mathrm{h}}u)\|_{0}\|\partial_{\mathrm{h}}^{j}\partial_{3}^{1-i}W_{3}\|_0
+\|(q,u,w)\|_{2}\mathfrak{D}.
\end{align*}
Substituting this estimate back leads to
\begin{align}\label{202604151037nmfg}
&\frac{1}{2}\frac{\mathrm{d}}{\mathrm{d}t}\int\bar{\rho}|\partial_{\mathrm{h}}^{i}\partial_{3}^{1-i}W_{3}|^2\mathrm{d}x
+\lambda_8\|\partial_{\mathrm{h}}^{i}\partial_{3}^{1-i}W_{3}\|_{0}^2
\lesssim\|\partial_{3}^{1-i}(u_{t},\nabla\partial_{\mathrm{h}}u)\|_{\underline{i},0}^2+\|(q,u,w)\|_{2}\mathfrak{D}
\end{align}
holds for some positive constant $\lambda_8$.

Now, combining \eqref{20260415nh} with \eqref{202604151037nmfg}, summing over $j$ from $0$ to $i$,
and invoking \eqref{11292030es} together with Young's inequality, we arrive at
\begin{align}\label{202604171038gh}
&\frac{\mathrm{d}}{\mathrm{d}t}\widetilde{\|\partial_{3}^{1-i}W\|}_{\underline{i},0}^2
+\lambda_9\|\partial_{3}^{1-i}W\|_{\underline{i},0}^2\nonumber\\[1mm]
&\lesssim \bar{\mathfrak{D}}
+  (1-i)\|W\|_{\underline{1},0}^2
+\|\nabla p\|_{1}\|u\|_{3}^{1/2}\|u\|_{\underline{2},1}^{1/2}
+\|(q,u,w)\|_{2}\mathfrak{D}
\end{align}
for some positive constant $\lambda_9$,
where $\widetilde{\|\partial_{3}^{1-i}W\|}_{\underline{i},0}^2$ is equivalent to $\|\partial_{3}^{1-i}W\|_{\underline{i},0}^2$.
Consequently, based on \eqref{202604171038gh}, there exist positive constants $d_{i}$ and a positive constant $\lambda_{10}$, such that
\begin{align}\label{20260417nc}
&\frac{\mathrm{d}}{\mathrm{d}t}\sum_{i=0}^{1}d_{i}\widetilde{\|\partial_{3}^{1-i}W\|}_{\underline{i},0}^2
+\lambda_{10}\sum_{i=0}^{1}d_{i}\|\partial_{3}^{1-i}W\|_{\underline{i},0}^2
\lesssim\bar{\mathfrak{D}}
+\|\nabla p\|_{1}\|u\|_{3}^{1/2}\|u\|_{\underline{2},1}^{1/2}
+\|(q,u,w)\|_{2}\mathfrak{D}.
\end{align}

Defining the functional
$$\widetilde{\|W\|}_1^2:=\sum_{i=0}^{1}d_{i}\widetilde{\|\partial_{3}^{1-i}W\|}_{\underline{i},0}^2.$$
It is clear that this functional is equivalent to $\|W\|_{1}^2$. Thus, the desired estimate \eqref{202604102028gh} follows immediately from \eqref{20260417nc}.
Finally, following the same arguments as in the derivation of \eqref{20260411nd}, by applying $\|\cdot\|_1^2$ to \eqref{202604101338nmes} and utilizing \eqref{12281530es} along with \eqref{aprpioseses} yields \eqref{20260411ndkl}, which completes the proof of Lemma \ref{uwnormalgh}.
\hfill$\Box$
\end{pf}

%%%%%%%%%%%%%%%%%%%%%%%%%%%%%%%%%%%%%%%%%%%%%%%%%%%%%%%%%%%%%%%%%%%%%%%
\subsection{A priori energy estimate}
We first provide the equivalent estimate of $\mathfrak{E}$.

%%%%%%%%%%%%%%%%%%%%%%%%%%%%%%%%%%%%%%%%%%%%%%%%%%%%%%%%%%%%%%%%%%%%%%%

\begin{lem}\label{lem:260415}
Under assumption \eqref{aprpioseses} with sufficiently small $\delta$, it holds that
\begin{align}\label{2604150855}
\mathfrak{E}\;\;\mbox{is equivalent to}\;\;\|(u,w)\|_{2}^2.
\end{align}
\end{lem}
\begin{pf}
Thanks to \eqref{11292030es} with $i=0$, we have the equivalence
$$\|(u,w)\|_2^2\lesssim  \mathfrak{E}\lesssim \|(u,w)\|_2^2+\|u_t\|_0^2+\|w_{t}\|_1^2.$$
Therefore,
to obtain \eqref{2604150855}, it suffices to verify
\begin{align}
\label{202604151637}
\|u_t\|_0^2+\|w_{t}\|_1^2\lesssim \|(u,w)\|_{2}^2.
\end{align}
To this end, we first invoke \eqref{0102nb}$_2$ and \eqref{12276213738es} with $i=1$ to deduce that
\begin{align*}
\|w_{t}\|_1^2\lesssim \|u\|_{2}^2+\|w\|_{1}^2+\|(q,u)\|_{2}^2\big(\|w\|_{2}^2+\|w_{t}\|_{1}^2\big),
\end{align*}
Under the smallness assumption \eqref{aprpioseses} on $\delta$, the higher-order term on the right-hand side can be absorbed, yielding
\begin{align}
\label{202604151638}
\|w_{t}\|_1^2\lesssim \|(u,w)\|_{2}^2.
\end{align}

We now turn to the estimate for $u_t$.
Taking the inner product of \eqref{0102nb}$_1$ with $u_t$ in $L^2$, integrating by parts over $\Omega$ and using \eqref{0102nb}$_3$--\eqref{0102nb}$_4$, we obtain
\begin{align*}
\int\bar{\rho} |u_t|^2\mathrm{d}x&=\int((\mu + \chi) \Delta u+2\chi \nabla \times w)\cdot u_t\mathrm{d}x+\int\mathcal{M}^3\cdot u_t\mathrm{d}x\nonumber\\[1mm]
&\lesssim\big(\|u\|_{2}+\|w\|_1+\|\mathcal{M}^3\|_{0}\big)\|u_t\|_0.
\end{align*}
Then, applying \eqref{1227213738es} with $i=0$ and \eqref{aprpioseses}, along with Young's inequality, leads to
\begin{align}\label{12052037}
\|u_{t}\|_{0}^2\lesssim\|u\|_{2}^2+\|w\|_1^2.
\end{align}
Combining \eqref{202604151638} and \eqref{12052037} yields the desired estimate \eqref{202604151637}, which completes the proof.
\hfill$\Box$
\end{pf}

Thanks to Proposition \ref{pro12281500es} and Lemmas \ref{lem:11292030es}--\ref{lem:260415} above, we now establish the following energy inequality.
\begin{pro}
\label{pro:0502n}
Under assumption \eqref{aprpioses} with sufficiently small $\delta$,
there exists a function $\tilde{\mathfrak{E}}$, which is equivalent to $\mathfrak{E}$, such that
\begin{align}   \label{emdsldsn37}
\frac{\mathrm{d}}{\mathrm{d}t} \tilde{\mathfrak{E}}+c\mathfrak{D}\leqslant0.
\end{align}
Additionally, it holds that
\begin{align}
\label{01021117n37}
\tilde{\mathfrak{E}}\lesssim\mathfrak{D}.
\end{align}
\end{pro}
\begin{pf}
To begin with, combining \eqref{12281503nmnmes} with \eqref{202604102028gh}--\eqref{20260411ndkl}, we deduce that
there exist a constant $\lambda_{11}>0$ and a suitably large constant $c_6$ such that
\begin{align}   \label{202604171036}
\frac{\mathrm{d}}{\mathrm{d}t} \mathfrak{E}_1+\lambda_{11}\mathfrak{D}_1\lesssim
\|\nabla p\|_{1}\|u\|_{3}^{1/2}\|u\|_{\underline{2},1}^{1/2}+\|(q,u,w)\|_{2}\mathfrak{D},
\end{align}
where we have defined that
\begin{align*}
&\mathfrak{E}_1:=\widetilde{\|W\|}_{1}^2+\|\mathrm{div}w\|_{1}^2+c_6\tilde{\bar{\mathfrak{E}}},
\quad\mbox{and}\quad\mathfrak{D}_1:=\|W\|_{1}^2+\|\mathrm{div}w\|_{1}^2+c_6\bar{\mathfrak{D}}.
\end{align*}
By virtue of \eqref{11292030es}--\eqref{12372030es} and \eqref{friedrich}, we have the following bounds:
\begin{align*}
&%\label{11292038es}
\|u\|_{2}^2+\|\nabla p\|_{0}^2\lesssim\|W\|_{0}^2+\bar{\mathfrak{E}},\\[1mm]
&%\label{12372038es}
\|u\|_{3}^2+\|\nabla p\|_{1}^2\lesssim\|W\|_{1}^2+\bar{\mathfrak{D}}.
\end{align*}
Furthermore, we also have
\begin{align*}
&%\label{01013630}
\|w\|_{2}\lesssim\|w\|_{\underline{2},0}+\|W\|_{1}+\|\mathrm{div}w\|_{1}.
\end{align*}
Collecting these estimates together with \eqref{2604150855} and \eqref{202604151638}, we conclude that the following equivalence and bounding relations hold:
\begin{align}
&\label{04171507}
\mathfrak{E}_1,\;\mathfrak{E}\;\mbox{and}\;\|(u,w)\|^2_2
\;\;\mbox{are equivalent},\\[1mm]
&\label{04171455}
\mathfrak{D}_1\;\;\mbox{is equivalent to}\;\;\mathfrak{D}\\[1mm]
&\label{01012021}
\tilde{\bar{\mathfrak{E}}},\;\mathfrak{E}_1\lesssim\mathfrak{D}.
\end{align}
Now, summing \eqref{12281503nmnmes} and \eqref{202604171036}, and invoking Young's inequality together with the smallness condition \eqref{aprpioseses}, we can find a suitably small constant $\lambda_{12}> 0$ and a suitably large constant $c_7$  such that
\begin{align}   \label{041717nml30}
\frac{\mathrm{d}}{\mathrm{d}t} \tilde{\mathfrak{E}}+\lambda_{12}(\mathfrak{D}_1+c_7\bar{\mathfrak{D}})\leqslant0,
\end{align}
where $\tilde{\mathfrak{E}}$ is defined by
\begin{align*}
&\tilde{\mathfrak{E}}:=\mathfrak{E}_1+c_7\tilde{\bar{\mathfrak{E}}}.
\end{align*}
It is evident from \eqref{04171507} that $\tilde{\mathfrak{E}}$ is also equivalent to $\mathfrak{E}$.
Consequently, the desired inequalities \eqref{emdsldsn37}--\eqref{01021117n37} follow immediately from \eqref{041717nml30} combined with \eqref{04171455}--\eqref{01012021}. This completes the proof.
\hfill$\Box$
\end{pf}

%%%%%%%%%%%%%%%%%%%%%%%%%%%%%%%%%%%%%%%%%
We are now in a position to derive the following \emph{a priori} energy estimate, which combines with the well-posedness result and a continuity argument,
yields Theorem \ref{thm2}.
\begin{thm}\label{thm2nm}
There exists a sufficiently small constant $\delta_2\in(0,1)$ such that if \eqref{aprpioseses} holds for $\delta\leqslant\delta_2$, then
\begin{align}
&\label{emdsldsa37nmn}
e^{ct}\mathfrak{E}(t)+\int_0^{t}e^{ct/2}\mathfrak{D}(\tau)\mathrm{d}\tau\lesssim \|(u^0,w^0)\|_{2}^2,\\[2mm]
&\label{emdsldsa38nmn}
\mathscr{E}(t)\lesssim \|( q^0, u^0, w^0)\|_{2}^2
\end{align}
hold for all $t\in[0,T]$.
\end{thm}
\begin{pf}
By virtue of \eqref{01012021},
it follows from \eqref{emdsldsn37} that there exist a sufficiently small constant $\tilde{\delta}_2\in(0,1)$ and a constant $\lambda_{13}> 0$ such that,
for any $\delta\leqslant \tilde{\delta}_2$,
\begin{align*}   %\label{emdsldsn15n37}
\frac{\mathrm{d}}{\mathrm{d}t} \tilde{\mathfrak{E}}+\lambda_{13}\tilde{\mathfrak{E}}\leqslant0,
\end{align*}
which, together with \eqref{04171507}, implies
\begin{align}
&\label{estemalasNnn}
\mathfrak{E}\lesssim\tilde{\mathfrak{E}}\lesssim\|(u^0,w^0)\|_{2}^2e^{-\lambda_{13}  t}.
\end{align}

Furthermore, we can deduce from \eqref{emdsldsn37} and \eqref{04171455}--\eqref{01012021} that there exist a sufficiently small constant $\hat{\delta}_2\in(0,1)$, a constant $\lambda_{14}> 0$ and a suitably large constant $c_8$ such that, for any $\delta\leqslant \hat{\delta}_2$,
\begin{align*}%\label{01021105n}
&\frac{\mathrm{d}}{\mathrm{d}t}\left(e^{\lambda_{14}t}\tilde{\mathfrak{E}}\right)
+ce^{\lambda_{14}t}\mathfrak{D}
\leqslant \lambda_{14}e^{\lambda_{14}t}\tilde{\mathfrak{E}}\leqslant c_{8}\lambda_{14}e^{\lambda_{14}t}\mathfrak{D}.
\end{align*}
By choosing $\lambda_{14}$ sufficiently small such that $\lambda_{14}<\lambda_{13}/2$ and $c_{8}\lambda_{14}<c/2$, we then obtain the time-decay estimate:
\begin{align}\label{01021105nn}
&e^{\lambda_{14}t}\mathfrak{E}+\int_{0}^{t}e^{\lambda_{14}\tau}\mathfrak{D}(\tau)\mathrm{d}\tau\lesssim \|(u^0,w^0)\|_{2}^2.
\end{align}

Additionally, invoking \eqref{0102n}$_1$ and \eqref{0102n}$_3$--\eqref{0102n}$_4$, we observe that
\begin{align}\label{04171205nn}
&\frac{1}{2}\frac{\mathrm{d}}{\mathrm{d}t}\|q\|_{2}^2%=\sum_{\alpha\in\mathbb{N}^3,\;|\alpha|=0}^2\int|\partial^{\alpha}q_t\partial^{\alpha}q|\mathrm{d}x
\lesssim\sum_{\alpha\in\mathbb{N}^3,\;|\alpha|=0}^2\int\left|[\partial^{\alpha},u]\cdot\nabla q\partial^{\alpha}q\right|\mathrm{d}x\lesssim\|q\|_{2}^2\|u\|_{3}.
\end{align}
Applying Gronwall's inequality to \eqref{04171205nn} and exploiting the bound in \eqref{01021105nn}, we find that there exists a sufficiently small constant $\bar{\delta}_2\in(0,1)$ such that, for any $\delta\leqslant \bar{\delta}_2$,
\begin{align}\label{04171236nn}
&\|q\|_{2}^2\lesssim\|q^0\|_{2}^2e^{\sqrt{\int_{0}^{t}e^{-\lambda_{14}\tau/2}\mathrm{d}\tau\int_{0}^{t}e^{\lambda_{14}\tau/2}\|u(\tau)\|_{3}^2\mathrm{d}\tau}}
\lesssim\|q^0\|_{2}^2e^{c\|(u^0, w^0)\|_{2}}\lesssim\|q^0\|_{2}^2.
\end{align}
Consequently, collecting \eqref{estemalasNnn}, \eqref{01021105nn} and \eqref{04171236nn}, and taking $\delta\leqslant \delta_2:=\min\{\tilde{\delta}_2, \hat{\delta}_2,\bar{\delta}_2\}$,  we arrive at \eqref{emdsldsa37nmn} and \eqref{emdsldsa38nmn} for $\delta\leqslant\delta_2$. This completes the proof.
\hfill$\Box$
\end{pf}

\vspace{5mm}
\noindent\textbf{Acknowledgements.}
The research of Youyi Zhao was supported by NSFC (12371233 and 12401289),
the Natural Science Foundation of Fujian Province of China (2024J08029),
and the Research Foundation of Fuzhou University (XRC-24050).
The author is deeply grateful to Prof. Fei Jiang for his useful discussions.

\vspace{5mm}
%%%%%%%%%%%%%%%%%%%%%%%%%%%%%%%%%%%%%%%%%%%%%%%%%%%%%%
\noindent\textbf{Conflict of Interest.}
The author states that there is no conflict of interest.

\vspace{5mm}

%%%%%%%%%%%%%%%%%%%%%%%%%%%%%%%%%%%%%%%%%%%%%%%%%%%%%%%%%%%%%%%%%%%%%%%%%%%%%%%%%%%%%%%%%%%%%%%%%%%%%%%%%%
\appendix
\section{Analysis tools}\label{sec:09}
%\section{Appendix}\label{appendix}
\renewcommand\thesection{A}
This Appendix is aims to provide some mathematical analysis tools, which have been used in the previous sections.
It should be noted that in this appendix we still adapt the simplified mathematical notations in Section \ref{main results}.
For the sake of simplicity, we still use the notation $a\lesssim b$ means that $a\leqslant cb$ for some constant $c>0$,
where the positive constant $c$ may depend on the physical parameters, but does not depend on the initial data or time, and may vary from line to line.

\begin{lem}\label{lem:10220826}
Embedding inequalities (see \cite[Theorem 4.12]{ARAJJFF}):
\begin{align}
&\label{embed237}
\|f\|_{L^p}\lesssim \| f\|_{1}\quad\mathrm{for}\;\;2\leqslant p\leqslant6,\\[1mm]
&\label{embed2}
\|f\|_{C^0(\bar{\Omega})}=\|f\|_{L^\infty}\lesssim \| f\|_{2}.
\end{align}
\end{lem}

\begin{lem}\label{10220830}
Friedrichs inequality (see \cite[Theorem 1.42]{NASII04}):
\begin{align}
\label{friedrich}
\|f\|_{0}\lesssim \|\nabla f\|_{0}\quad\mathrm{ for }\;f\in H^1_0.
\end{align}
\end{lem}

\begin{lem}\label{lem:10220837nm}
Young's inequality (see \cite{ELGP}): Let $1<p, q<\infty$, $\frac{1}{p} + \frac{1}{q} = 1$. Then
\begin{align}\label{young}
    ab \leqslant \frac{a^p}{p} + \frac{b^q}{q}\;\;\;(a,b>0).
\end{align}
A more general form, known as Young's inequality with $\epsilon$, states that %for any $\epsilon>0$,
\begin{align}\label{eyoung}
    ab \leqslant \epsilon a^p + C(\epsilon) b^q\;\;\;(a,b>0,\;\epsilon>0),
\end{align}
where $C(\epsilon)=(\epsilon p)^{-p/q} q^{-1}$.
\end{lem}
\begin{rem}
When $p = q = 2$, inequality \eqref{eyoung} reduces to the form
$$ab\leq \epsilon a^2+\frac{1}{4\epsilon}b^2\;\;\;(a,b>0,\;\epsilon>0),$$
which is widely used in energy estimates.
\end{rem}

\begin{lem}\label{10220828}
Interpolation inequality
(see \cite[Theorem 5.2]{ARAJJFF}):
For any $0\leqslant j< i$, $\epsilon>0$, % and $f\in H^{i}$,
\begin{align}\label{interpolation}
&\|f\|_{j}\lesssim\|f\|_{0}^{1-{j}/{i}}\|f\|_{i}^{{j}/{i}}
\leqslant {C}( i,j,\epsilon)\|f\|_{0} +\epsilon\|f\|_{i}\quad\mathrm{for}\;f\in H^{i},
\end{align}
where the constant ${C}( i,j,\epsilon)$ depends $i,j$ and $\epsilon$, and Young's inequality
has been used in the last inequality.
\end{lem}

%We will need some estimates for the product of functions in Sobolev spaces (referred to as product estimates).
\begin{lem}\label{10220833}
Product estimates (\cite[Lemma A.9]{JFJSWZhangwei}): Let $0\leqslant i\leqslant 2$, for $f,g\in H^{i}$, we have
\begin{align}
&\label{product}
 \|fg\|_{i}\lesssim    \begin{cases}
 \|f\|_{1}\|g\|_{1} & \quad\mathrm{for}\;\;i=0;  \\[1mm]
 \|f\|_{i}\|g\|_{2} & \quad\mathrm{for}\;\;0\leqslant i\leqslant 2.
                    \end{cases}
\end{align}
\end{lem}
%\begin{pf}
%Please refer to \cite[Lemma A.9]{JFJSWZhangwei} for the proof of \eqref{product}.
%\hfill$\Box$
%\end{pf}

\begin{lem}\label{lem:11190836}
Dual estimate (see \cite[Lemma A.8]{JFJHJS}): For $\varphi$, $\psi\in H^{1/2}(\mathbb{R}^2)$ and $\partial_{\mathrm{h}}\varphi\in L^1(\mathbb{R}^2)$,
\begin{align}\label{11190840}
\left|\int_{\mathbb{R}^2}\partial_{\mathrm{h}}\varphi\psi\mathrm{d}x_{\mathrm{h}}\right|
\lesssim|\varphi|_{1/2}|\psi|_{1/2}.
\end{align}
%Moreover, we have
%\begin{align}\label{11190841}
%|\partial_{\mm{h}}\varphi|_{-1/2}
%\lesssim|\varphi|_{1/2}.
%\end{align}
\end{lem}
%\begin{pf}
%Please refer to \cite[Lemma A.8]{JFJHJS} for the proof.
%\hfill$\Box$
%\end{pf}

\begin{lem}\label{10220835bnbm}
Trace estimate(see \cite[Lemma A.6]{JFJHJS}): For $f \in H^{1+i}$ with $i\geqslant0$, it holds that
\begin{equation}\label{37190928}
\| f|_{y_3=a} \|_{H^{i+1/2}(\mathbb{R}^2)} %\lesssim \| f \|_{\underline{1+i},0}^{1/2} \| \partial_3 f \|_{\underline{i},0}^{1/2} + \| f \|_{\underline{1+i},0}
\lesssim \| f \|_{\underline{1+i},0}^{1/2} \| f \|_{\underline{i},1}^{1/2}\quad\mathrm{for\;any }\;\;a\in[0,1].
\end{equation}
\end{lem}
%\begin{pf}
%Please refer to \cite[Lemma A.6]{JFJHJS} for the proof.
%\hfill$\Box$
%\end{pf}

\begin{lem}\label{prosfxdxx}
Elliptic estimate (see \cite{SAADLNEAE1964}): Let $k\geqslant 0$ and $f^1\in H^{k}$, then there exists a unique solution $\varphi\in H^{k+2}$ to the following Lam\'e problem:
\begin{equation*}
\begin{cases}
-\mu\Delta \varphi-(\mu+\mu')\nabla \mathrm{div}\varphi=f^1\quad&\mathrm{ in }\;\;\Omega, \\
\varphi=0 &\mathrm{ on }\;\;\partial\Omega.
\end{cases}
\end{equation*}
Moreover, the solution enjoys
\begin{equation}
\label{Ellipticestimate}
\|\varphi\|_{k+2}\lesssim\|f^1\|_{k}.
\end{equation}
\end{lem}

\begin{lem}\label{10220835}
Stokes estimate:
(see \cite[Lemma A.12]{WYJAIM2020}):
Let $k \geqslant 0$. Suppose that $f^2 \in H^{k}$ and $ f^3 \in H^{k+1}$, then $\varphi \in H^{k+2}, \nabla \psi \in H^{k}$ solving the problem
\begin{equation*}%\label{onelayerstokes}
\begin{cases}
\nabla \psi-\mu\Delta \varphi= f^{2}\quad &\mathrm{ in }\;\; \Omega,\\[0.5mm]
\mathrm{div}\varphi=f^{3} &\mathrm{ in }\;\;  \Omega,\\[0.5mm]
\varphi=0 &\mathrm{ on }\;\; \partial \Omega.
\end{cases}
\end{equation*}
Moreover, the solution enjoys that
\begin{align}\label{11190928}
\|\varphi\|_{k+2}+\|\nabla \psi\|_{k}\lesssim
\|\varphi\|_0^2 +\|f^{2}\|_{k}+\|f^{3}\|_{k+1}.
\end{align}
\end{lem}

\renewcommand\refname{References}
\renewenvironment{thebibliography}[1]{%
\section*{\refname}
\list{{\arabic{enumi}}}{\def\makelabel##1{\hss{##1}}\topsep=0mm
\parsep=0mm
\partopsep=0mm\itemsep=0mm
\labelsep=1ex\itemindent=0mm
\settowidth\labelwidth{\small[#1]}%
\leftmargin\labelwidth \advance\leftmargin\labelsep
\advance\leftmargin -\itemindent
\usecounter{enumi}}\small
\def\newblock{\ }
\sloppy\clubpenalty4000\widowpenalty4000
\sfcode`\.=1000\relax}{\endlist}
\bibliographystyle{model1b-num-names}
%\bibliography{refs}

\end{document}